\documentclass[11pt,a4paper]{article}
\usepackage{geometry}
\usepackage{graphicx}
\usepackage{bm,array}
\usepackage{comment}
\usepackage{caption}
\usepackage{subcaption}
\usepackage{tablefootnote}
 \usepackage{makecell}
\usepackage{booktabs}
 \usepackage{multirow}
 \usepackage{float}
 \usepackage{algorithmic}
 \usepackage{tikz}
 \usetikzlibrary{arrows.meta, positioning, shapes.geometric, shadows}
\usepackage[normalem]{ulem}

\usepackage[pdfpagemode={UseOutlines},bookmarks=true,bookmarksopen=true,
bookmarksopenlevel=0,bookmarksnumbered=true,hypertexnames=false,
colorlinks,linkcolor={blue},citecolor={blue},urlcolor={red},
pdfstartview={FitV},unicode,breaklinks=true]{hyperref}
\hypersetup{urlcolor=blue, colorlinks=true}

\usepackage{setspace}
\usepackage{vmargin}
\setmarginsrb{ 1.0in}  % left margin
{ 0.6in}  % top margin
{ 1.0in}  % right margin
{ 0.8in}  % bottom margin
{  20pt}  % head height
{0.25in}  % head sep
{9pt}  % foot height
{ 0.3in}  % foot sep

\usepackage{authblk}

\usepackage{amsmath}
\usepackage{amsfonts}
\usepackage{amssymb}
\usepackage[round, sort,comma,authoryear]{natbib}

\newtheorem{theorem}{Theorem}

\newtheorem{proposition}{Proposition}
\newtheorem{remark}{Remark}

\newenvironment{proof}[1][Proof]{\noindent\textbf{#1.} }{\ \rule{0.5em}{0.5em}}

\usepackage[ruled,vlined]{algorithm2e}

\usepackage{multirow}% http://ctan.org/pkg/multirow
\usepackage{hhline}% http://ctan.org/pkg/hhline
\usepackage{footnote}
\usepackage[ruled,vlined]{algorithm2e}
\usepackage{algorithmic}

\newcommand{\cP}{{\mathcal{P}}}

\newcommand{\by}{\textbf{y}}
\newcommand{\bs}{\textbf{s}}

\newcommand{\bu}{\textbf{u}}

\newcommand{\bp}{\textbf{p}}
\newcommand{\ba}{\textbf{a}}
\newcommand{\bv}{\textbf{v}}
\newcommand{\bz}{\textbf{z}}

\newcommand{\btheta}{\pmb{\theta}}

\usepackage{mathtools}

\usepackage{mathtools}

\newif\ifnotes\notestrue
\def\htien#1{}

\usepackage{xcolor}

\begin{document}
%%%%%%%%%%%%%%%%

%    \RUNTITLE{Joint Facility Location and Cost Optimization}

% \TITLE{Bundling Information Goods of Decreasing Value}
%\TITLE{Joint  Location and Cost Planning in Maximum Capture Facility Location under Multiplicative Random Utility Maximization}

%should be entered in ONE field, separated by a comma.
%\EMAIL field can be repeated if more than one author
%\ARTICLEAUTHORS{%
%\AUTHOR{Ngan Ha Duong, Tien Thanh Dam, Thuy Anh Ta}
%\AFF{Department of Computer Science, Phenikaa University}
%\AUTHOR{Tien Mai}
%\AFF{School of Computing and Information Systems, Singapore Management University, \EMAIL{atmai@smu.edu.sg}}
%} % end of the block

%\ABSTRACT{%
%...
%}%

%\KEYWORDS{deterministic inventory theory; infinite linear programming duality;
%  existence of optimal policies; semi-Markov decision process; cyclic schedule}

%\KEYWORDS{Maximum capture, cost optimization,  multiplicative random utility, mixed-integer program} 
%\HISTORY{This paper wasfirst submitted on April 12, 1922 and has been with the authors for 83 years for 65 revisions.}

%\maketitle
\newcolumntype{C}{>{\centering\arraybackslash}p{4em}}

\title{\textbf{Constrained Pricing under Finite Mixtures of Logit}}
\author[1]{Hoang Giang Pham}
%\author[1]{Tien Thanh Dam}
%\author[1]{Thuy Anh Ta}
\author[1]{Tien Mai}
\affil[1]{\it\small
School of Computing and Information Systems, Singapore Management University}
%\affil[4]{\it\small Corresponding author}

\date{}
\maketitle

% \pagebreak
% \begin{center}
% 	\linespread{1.5}
% 	\LARGE
% 	\bf
% 	{Joint  Location and Cost Planning in Maximum Capture Facility Location under Multiplicative Random Utility Maximization}
% \end{center}

\begin{abstract}
The mixed logit model is a flexible and widely used demand model in pricing and revenue management. However, existing work on mixed-logit pricing largely focuses on unconstrained settings, limiting its applicability in practice where prices are subject to business or regulatory constraints. We study the \emph{constrained pricing problem} under multinomial and mixed logit demand models.
For the multinomial logit model, corresponding to a single customer segment, we show that the constrained pricing problem admits a \emph{polynomial-time approximation scheme (PTAS)} via a reformulation based on \emph{exponential cone programming}, yielding an $\varepsilon$-optimal solution in polynomial time. For finite mixed logit models with $T$ customer segments, we reformulate the problem as a \emph{bilinear exponential cone program} with $O(T)$ bilinear terms. This structure enables a \emph{branch-and-bound} algorithm whose complexity is exponential only in $T$. Consequently, constrained pricing under finite mixtures of logit admits a PTAS when the number of customer segments is bounded. Numerical experiments demonstrate strong performance relative to state-of-the-art baselines.
\end{abstract}

\noindent
\textbf{Notation:}
Boldface characters represent matrices (or vectors), and $a_i$ denotes the $i$-th element of vector $\ba$. We use $[m]$, for any $m\in \mathbb{N}$, to denote the set $\{1,\ldots,m\}$.

\section{Introduction}
\subsection{Motivation} Pricing is a central lever in revenue management and plays a decisive role in shaping firms’ profitability and market performance \citep{talluri2004,phillips2005}. By directly influencing both demand levels and substitution patterns across products, pricing decisions affect not only short-term revenue but also long-term competitive positioning \citep{gallego1994,anderson1992}. As a result, the effectiveness of any pricing strategy critically depends on the underlying demand model’s ability to accurately capture customer behavior in response to price changes \citep{hanson1996_ms,train2009,guadagni2008_mkts}.

Discrete choice models have become the dominant framework for modeling customer purchase behavior in pricing and assortment optimization \citep{talluri2004,li2011_msom,Rusmevichientong2014_pom}. Among these, the mixed logit  model has emerged as a particularly powerful and widely used approach, as it captures customer heterogeneity by allowing for multiple latent customer segments with distinct preference structures \citep{mcfadden2000,ruben2022mnsc}. By relaxing the restrictive homogeneity assumptions of simpler models, the mixed logit model provides a significantly more realistic representation of observed purchasing behavior and is now commonly employed in both academic research and practical revenue management applications \citep{train2009,Hensher2003ThePractice}.

Despite the appeal of rich choice models, the resulting pricing optimization problems are notoriously challenging. Even in the simplest setting of a multinomial logit (MNL) model with a single customer segment, the revenue function is generally nonconcave in prices, complicating global optimization \citep{li2011_msom}. The difficulty is further exacerbated under more expressive models such as nested logit or mixed logit, where the revenue landscape may contain multiple local optima \citep{Li2019_msom}. These nonconvexities make pricing optimization particularly challenging when prices are subject to practical constraints—such as bounds, capacity constraints, or business rules—since standard gradient-based or local-search methods may converge to highly suboptimal solutions. Most existing work in the pricing literature addresses this challenge by focusing either on relatively simple choice models, most notably the MNL model, or on unconstrained pricing settings. While these approaches yield valuable theoretical insights, they leave a significant gap between theory and practice. In real-world applications, pricing decisions are almost always constrained by operational, regulatory, or managerial considerations, and customer demand is rarely well approximated by a single homogeneous segment. As a consequence, there is currently a lack of efficient and provably reliable methods for solving constrained pricing problems under rich demand models such as the mixed logit model.

\subsection{Technical Contributions.}
In this paper, we address this gap by developing \emph{novel PTASs} for constrained price optimization under MNL and finite-mixtures of logit (FMNL) demand models. Our results provide the first global optimization guarantees for constrained pricing problems under rich choice models that capture customer heterogeneity. Specifically, we make the following contributions:
\begin{itemize}
\item \textit{Constrained pricing under the MNL model.}
We first study pricing under the MNL model, one of the most widely used choice models in practice. While existing work on pricing under MNL demand either focuses primarily on the unconstrained setting on prices \cite{li2011_msom,zhang2018_opre}, or relies on mixed-integer linear programming (MILP)–based formulations that do not admit polynomial-time guarantees \citep{Shao2026ConstrainedManagement}, we show that the constrained pricing  problem (with bound and linear constraints on the prices) admits a \emph{PTAS}. Our approach is based on a novel convex exponential-cone reformulation combined with a bisection procedure, yielding global approximation guarantees despite the nonconcavity of the revenue function in prices.

\item \textit{Constrained pricing under FMNL models.}  
We then study constrained pricing under FMNL models, which capture customer heterogeneity through multiple latent segments and are significantly more expressive than MNL. We reformulate the problem as a \emph{bilinear--convex program} in which all nonconvexities are isolated into a small number of bilinear terms—one per customer segment. Exploiting this structure, we propose a tailored branch-and-bound (B\&B) algorithm whose worst-case complexity is exponential only in the number of segments and polynomial in all other problem dimensions. To the best of our knowledge, this is the \emph{first PTAS for constrained pricing under FMNL models}.

\item \textit{Computational experiments.}  
We conduct extensive numerical experiments to evaluate the performance of the proposed algorithms. The results demonstrate that our methods not only provide strong theoretical guarantees but also consistently outperform state-of-the-art baselines in terms of solution quality and robustness, particularly in settings with practical pricing constraints.
\end{itemize}

In summary, we develop the first PTAS algorithms for constrained pricing under FMNL demand. Our approach achieves provable polynomial-time approximation guarantees while simultaneously delivering superior empirical performance, thereby bridging a critical gap between practical pricing requirements and algorithmic tractability.

\subsection{Literature Review}

%Pricing is a central problem in revenue management, as prices directly shape demand levels, substitution patterns, and ultimately firm profitability. The effectiveness of pricing decisions therefore hinges on the ability of demand models to accurately predict customer responses to price changes. Discrete choice models have become the dominant framework for this purpose, as they provide a principled representation of individual utility maximization and aggregate demand behavior. Among these models, the mixed logit (or finite-mixture logit, FMNL) model has gained particular prominence due to its ability to capture heterogeneity in customer preferences, which is critical for realistic pricing and revenue management applications.

In discrete choice modeling, the MNL model \cite{Luce1959,train2009,guadagni2008_mkts} remains the most widely used customer choice model in both academic and practitioner settings. In the context of price optimization, \cite{hanson1996_ms} demonstrated via counterexample that the MNL profit function is not quasiconcave in prices, and proposed a path-following heuristic to address this challenge. Subsequent work showed that the difficulty can be mitigated by reformulating the problem in terms of purchase probabilities (market shares), under which the objective becomes concave \citep{dong2009_msom,li2011_msom,song2021_ssrn}. This observation enables efficient optimization, since prices and market shares are in one-to-one correspondence. Recent work by \cite{Shao2026ConstrainedManagement} develops a MILP reformulation for the constrained pricing problem under the MNL model. While this approach is flexible and can accommodate general linear constraints on prices, solving MILPs is in general computationally intractable in the worst case and does not admit polynomial-time performance guarantees. As a result, such methods may face significant scalability challenges as the problem size increases.

Beyond MNL, richer discrete choice models have been proposed to capture more complex substitution patterns. The nested logit model introduces a hierarchical choice structure in which consumers first select a nest of products before choosing a specific item. Under constant price sensitivities within nests, the profit function remains concave in purchase probabilities \cite{li2011_msom}. When price sensitivities vary across products, \cite{gallego2014_opre} show that an adjusted markup remains constant within each nest. Extensions to multi-level nested structures have also been studied \citep{li2015_opre,huh2015_opre}, although determining appropriate nest structures in practice remains challenging \cite{koppelman2006}. Other choice frameworks considered in the pricing literature include the exponomial model \cite{alptekinouglu2016_opre}, paired combinatorial logit \cite{li2017_opre}, generalized extreme value models \cite{zhang2018_opre}, and Markov chain choice models \cite{dong2019_pom}.

Despite these advances, most pricing models beyond MNL rely on restrictive assumptions of consumer homogeneity. With the notable exception of \cite{hanson1996_ms}, these works typically assume identical preferences or price sensitivities across the population. This simplification is well known to be problematic: ignoring heterogeneity leads to biased parameter estimates \cite{hsiao2022analysis} and systematically suboptimal pricing decisions, a concern already emphasized in early empirical studies such as \cite{guadagni2008_mkts}.

The FMNL model addresses this limitation by representing the population as a mixture of latent customer segments, each governed by its own MNL parameters. While FMNL models are widely used for demand estimation and assortment optimization, the literature on FMNL-based price optimization remains limited \citep{hanson1996_ms,Li2019_msom,ruben2022mnsc}. A key difficulty is that the FMNL revenue function is generally neither concave nor quasiconcave, even when expressed in terms of purchase probabilities \cite{Li2019_msom}. As a result, most existing approaches rely on local search, gradient-based heuristics, or bisection methods that lack global optimality guarantees, particularly in the presence of price constraints.  Recent work has begun to explore constrained pricing under richer discrete choice models. In particular, \cite{pham2026constrained} study constrained assortment and price optimization under the generalized nested logit model and propose a global optimization approach for settings with heterogeneous customer segments. Their method is based on piecewise linear approximation and mixed-integer nonlinear programming (MINLP). While flexible, such formulations are generally computationally demanding and do not admit polynomial-time performance guarantees. Most closely related to our work is \cite{ruben2022mnsc}, who propose a global optimization framework for unconstrained pricing under FMNL demand and establish a PTAS whose complexity is polynomial in the number of products and exponential only in the number of customer segments. While this represents a significant theoretical advance, their approach does not address constrained pricing, which is ubiquitous in practice due to capacity limits, pricing rules, and regulatory or managerial considerations.

In contrast to the existing literature, our work focuses explicitly on \emph{constrained} pricing under both MNL and FMNL demand. We develop approximation algorithms with global performance guarantees that remain polynomial-time for a fixed number of customer segments, even in the presence of general linear price constraints. By doing so, our results bridge a critical gap between expressive demand modeling and tractable, globally reliable pricing optimization methods.

\subsection{Managerial Insights.}
Our results offer several insights for managers responsible for pricing decisions in environments with heterogeneous customers and operational constraints.

First, our findings highlight the \textit{risk of relying on local or heuristic pricing methods} when customer heterogeneity is present. Even with a small number of customer segments, the revenue function under FMNL demand can exhibit multiple local optima. As a result, gradient-based or locally convergent methods may yield pricing policies that are significantly suboptimal. This underscores the importance of using globally reliable optimization tools when pricing decisions have material financial impact.

Second, the results emphasize that \textit{pricing constraints should be treated as a first-order modeling consideration rather than an afterthought}. Common practices such as projecting unconstrained optimal prices onto feasible regions can lead to substantial revenue loss. Our framework shows that explicitly accounting for constraints—such as price bounds, budget limits, or portfolio-level pricing rules—during optimization can materially improve performance relative to unconstrained or heuristic approaches.

Third, the paper demonstrates that \textit{rich demand models and computational tractability need not be mutually exclusive}. While FMNL models are often perceived as too complex for practical pricing optimization, our results show that when the number of customer segments is moderate, firms can obtain provably near-optimal prices in polynomial time. This suggests that managers can adopt more expressive demand models without sacrificing decision reliability, provided that the optimization is designed to exploit problem structure.

Finally, our findings suggest that \textit{managerial effort should focus on identifying a small number of meaningful customer segments}, rather than over-segmenting demand. Since the computational complexity scales primarily with the number of segments, a parsimonious segmentation that captures key behavioral differences can deliver most of the benefits of heterogeneous pricing while maintaining computational efficiency.

Overall, the paper provides a framework that enables managers to make pricing decisions that are both operationally feasible and globally near-optimal, even in the presence of complex demand interactions and practical business constraints.

\noindent\textbf{Paper Outline:}
The remainder of the paper is organized as follows. 
Section~\ref{sec:formulation} introduces the constrained pricing problem under MNL and FMNL demand models. 
Section~\ref{sec:pricing-MNL} studies constrained pricing under the MNL model and presents a PTAS based on an exponential-cone reformulation and a bisection scheme. 
Section~\ref{sec:pricing-FMNL} extends the analysis to FMNL models, develops a bilinear--convex reformulation, and proposes a global B\&B algorithm with complexity guarantees. 
Section~\ref{sec:experiments} reports numerical experiments. 
Section~\ref{sec:concl} concludes. 
The appendix contains proofs omitted from the main text and additional experimental results.

\section{Problem Formulations}\label{sec:formulation}
We consider the constrained pricing problem for a fixed assortment of products. The objective is to determine optimal prices for the offered products so as to maximize expected revenue, where customer demand is modeled using a FMNL model.
\subsection{Model Setup}
Let $[m]=\{1,2,\ldots,m\}$ denote the set of products. Each product $i \in [m]$ is associated with a decision variable $p_i \in [L_i,U_i]$, representing its price within lower and upper bounds $L_i$ and $U_i$. The deterministic component of the utility of product $i$ is assumed to be affine in price, i.e.,
$u_{ti}(p_i) = a_{ti} - b_{i} p_i,$
where $b_{i}>0$ denotes the price sensitivity of customer type $t \in [T]$ for product $i$ and  $a_{ti}$ captures other factors that affect customers choice probabilities.
We assume that price sensitivity parameters are independent of the customer segment. This assumption is critical for the tractability of the pricing problem and underpins the PTAS developed in this paper. When price sensitivities vary across segments, the resulting revenue function exhibits significantly more complex nonconvex interactions, which preclude efficient global approximation with current techniques. The same assumption is adopted in \cite{ruben2022mnsc}, where the authors emphasize its necessity for establishing PTAS guarantees even in the unconstrained pricing setting. Our work inherits this structural requirement and demonstrates that, under this assumption, global approximation guarantees can be extended to constrained pricing problems.

In the FMNL model, each customer belongs to one of $T$ segments, with probabilities $d_t>0$ satisfying $\sum_{t=1}^T d_t=1$. Conditional on type $t$, the probability that the customer purchases product $i$ at price vector $\bp=(p_1,\ldots,p_m)$ is: 
\[
P_t(i \mid \bp) = \frac{\exp(u_{ti}(p_i))}{1+\sum_{j\in[m]} \exp(u_{tj}(p_j))},
\]
where the denominator includes the ``no-purchase'' option with normalized utility $0$. The unconditional choice probability under the FMNL model is
$P(i \mid \bp) = \sum_{t=1}^T d_t P_t(i \mid \bp).$
The expected revenue from offering all products at prices $\bp$ is
\begin{align}
R(\bp) &= \sum_{i\in[m]} p_i \, P(i \mid \bp) =  \sum_{t=1}^T d_t \frac{ \sum_{i\in[m]} p_i\exp(a_{ti}-b_{i}p_i)}{1+\sum_{j\in[m]} \exp(a_{tj}-b_{j}p_j)}.
\label{eq:MXL-revenue}
\end{align}
% The objective function of the pricing problem under the FMNL model is highly nonconvex in prices due to its inherent sum-of-ratios structure. The expected revenue is a weighted sum of fractional terms, each corresponding to a customer segment, where both the numerator and denominator depend nonlinearly on prices. This structure destroys the unimodality properties present in the single-segment MNL case \citep{li2011_msom} and generally leads to a revenue landscape with multiple local optima.

\subsection{Constrained Price Optimization Problem}
While much of the existing literature focuses on unconstrained pricing problems, such formulations often fall short in real-world business settings, where pricing decisions are rarely made in isolation. In practice, firms operate under a wide range of constraints driven by operational, regulatory, contractual, and managerial considerations. For instance, lower and upper bounds on prices are commonly imposed to ensure minimum profit margins, comply with price-floor regulations, or maintain brand positioning and customer perceptions. Excessively low prices may erode brand value, while excessively high prices may violate fairness or anti-gouging regulations, especially in regulated industries such as utilities, healthcare, and transportation.

Beyond individual price bounds, pricing decisions across products are frequently coupled through aggregate or portfolio-level constraints. Retailers may impose budget or capacity constraints that limit the total revenue-weighted or volume-weighted prices across a product assortment. In subscription-based or bundled services, firms often require prices to satisfy linear relationships to ensure internal consistency across plans or tiers. Similarly, platform operators and regulators may enforce caps on weighted average prices, where weights reflect market shares, production costs, or social priorities. Such constraints naturally arise in revenue management settings with shared resources, such as airline seat capacity, hotel room inventory, or cloud computing infrastructure.

To capture these practically relevant considerations, we model feasible prices using a general set of linear constraints:
\begin{equation}  \label{eq:p-constraints}
\mathcal{P} = \left\{ \mathbf{p} \in \mathbb{R}^m \;\middle|\;
\begin{aligned}
& L_i \leq p_i \leq U_i, && \forall i \in [m], \\
& \sum_{i=1}^m \alpha_{ki} p_i \leq \beta_k, && \forall k \in [K]
\end{aligned}
\right\}.
\end{equation}
where $[K] = \{1,\ldots,K\}$ and  the coefficients satisfy $\alpha_{ki} \geq 0$ for all $i\in [m], k\in [K]$. The price optimization problem under the FMNL model is then formulated as
\begin{align}
\max_{\bp \in \cP} \quad & \left\{R(\bp) = \sum_{i\in[m]} p_i \, P(i \mid \bp) =  \sum_{t=1}^T d_t \frac{ \sum_{i\in[m]} p_i\exp(a_{ti}-b_{i}p_i)}{1+\sum_{j\in[m]} \exp(a_{tj}-b_{j}p_j)} \right\}.
\tag{\sf Price-FMNL} \label{prob:Price-MXL}
\end{align}

Problem~\eqref{prob:Price-MXL} is highly nonlinear and, in general, non-concave. The mixture across multiple customer types $t \in [T]$ breaks the convenient unimodality structure present in the standard MNL model \citep{li2011_msom}. For the FMNL model, $R(\bp)$ may contain multiple local maxima, and gradient-based methods can converge to suboptimal solutions.  
Intuitively, different customer segments may favor different price levels due to heterogeneous valuations, creating competing incentives in the objective function. A price that is nearly optimal for one segment may significantly reduce revenue from another, causing the aggregate revenue function to exhibit multiple peaks corresponding to different segment-driven trade-offs. As a result, the superposition of segment-level revenue curves can produce several locally optimal price points. Figure~\ref{fig:FMNL_example} illustrates this behavior in a simple setting with a single product ($m=1$). As the number of customer segments increases ($T=2,3,4$), the revenue function becomes increasingly multi-modal, highlighting the limitations of standard gradient-based local search approaches.

\begin{figure}[htb]
    \centering
    \includegraphics[width=\linewidth]{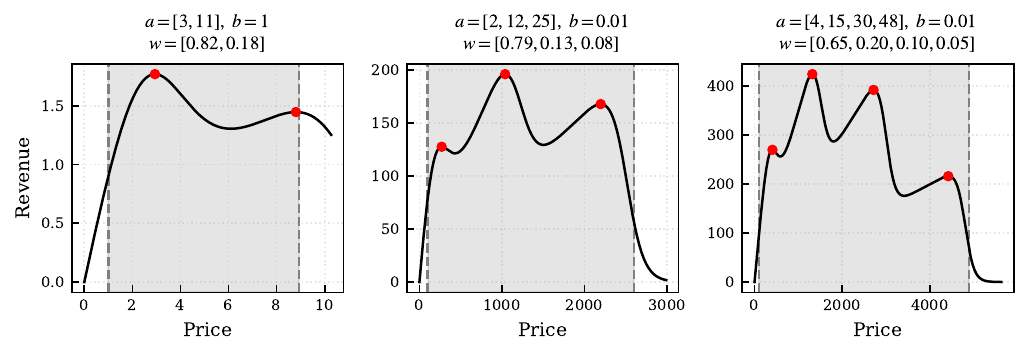}
    \caption{Revenue function under the FMNL model when $m=1$ and $T \in \{2,3,4\}$.}
    \label{fig:FMNL_example}
\end{figure}

\section{PTAS for Constrained Pricing under MNL}\label{sec:pricing-MNL}
In this section, we study the constrained pricing problem under the MNL demand model. While pricing under MNL demand has been extensively studied in the literature and is not the primary focus of this work, its relatively simple structure compared to FMNL models allows us to derive sharp algorithmic insights. We show that, when prices are subject to  linear constraints in \eqref{eq:p-constraints}, the constrained MNL-based pricing problem admits a PTAS based on a convex reformulation and a bisection procedure. This analysis serves as a foundation for the more general FMNL setting studied in subsequent sections.

\subsection{PTAS via Bisection and Exponential Cone Programming}
We now describe our approach for constructing a PTAS for the pricing problem under MNL, based on a bisection procedure and exponential cone programming (ECP). The key idea is to reformulate the constrained MNL pricing problem as a parametric feasibility problem indexed by a scalar revenue parameter. For any fixed parameter value, feasibility can be checked by solving a convex optimization problem that admits an ECP. By exploiting the monotonicity of the feasibility condition with respect to the parameter, we apply a bisection scheme to iteratively refine the parameter value and recover an $\varepsilon$-optimal solution. This combination of convex reformulation and bisection yields a globally reliable approximation algorithm with polynomial-time complexity.

Under the MNL model, the constrained pricing problem can be stated as follows:
\begin{align*}
     \max_{\bp \in \cP} \quad & R^{\text{MNL}}(\bp) =  
     \frac{ \sum_{i\in[m]} p_i \exp(a_{i}-b_{i}p_i)}{1+\sum_{j\in[m]} \exp(a_{j}-b_{j}p_j)}.
\end{align*}
Since the objective function under the MNL model is fractional in prices, we leverage a standard Dinkelbach-type (or fractional programming) transformation \citep{Dinkelbach1967} to convert the original problem into a parametric feasibility problem. Specifically, the constrained pricing problem can be reformulated as a univariate optimization problem over a scalar parameter $\theta$:
\[
\max_{\theta>0} \left\{ \theta \;\middle|\; \max_{\bp\in \mathcal{P}} \, G(\bp,\theta)\geq 0 \right\},
\]
where
\[
G(\bp,\theta) 
= \sum_{i\in[m]} p_i \exp(a_i-b_i p_i)
- \theta\left(1+\sum_{j\in[m]} \exp(a_j-b_j p_j)\right).
\]
Let
$\phi(\theta) := \max_{\bp\in\mathcal{P}} G(\bp,\theta)$
denote the optimal value of this inner problem. It is straightforward to verify that $\phi(\theta)$ is continuous and strictly decreasing in $\theta$. Intuitively, increasing $\theta$ penalizes the denominator term more heavily, thereby reducing the maximum achievable value of $G(\bp,\theta)$. As a result, the problem
$\max_{\theta>0} \left\{ \theta \;\middle|\; \phi(\theta)\geq 0 \right\}$
can be solved efficiently via a bisection (binary search) procedure over $\theta$. At each iteration, feasibility is checked by solving the sub-problem $ \max_{\bp\in\mathcal{P}} G(\bp,\theta)$.
\vspace{0.5em}

\noindent\textbf{Bisection.} We briefly describe the bisection procedure used to solve the constrained pricing problem under the MNL model. The algorithm starts from an initial interval $[\theta_{\min}, \theta_{\max}]$ that is guaranteed to contain the optimal value. In practice, these bounds can be chosen tightly: a valid lower bound $\theta_{\min}$ can be obtained from any feasible pricing solution (e.g., a heuristic or uniform pricing policy), while an upper bound $\theta_{\max}$ can be derived from simple relaxations of the problem, such as ignoring price constraints or bounding prices by $U_i$. Tighter initial bounds substantially reduce the number of bisection iterations required. At each iteration, we set $\theta$ to the midpoint of the current interval and solve the associated subproblem $\max_{\bp \in \mathcal{P}} G(\bp,\theta)$. If the optimal value of this subproblem is nonnegative, then $\theta$ is feasible and the lower bound is updated by setting $\theta_{\min} \leftarrow \theta$; otherwise, $\theta$ is infeasible and the upper bound is updated by setting $\theta_{\max} \leftarrow \theta$. By the strict monotonicity of the value function $\Phi(\theta)$, this procedure correctly maintains feasibility of the interval endpoints. The algorithm terminates when $\theta_{\max}-\theta_{\min}\leq \varepsilon$, at which point $\theta_{\min}$ is returned as an $\varepsilon$-optimal solution. Since the interval length is halved at each iteration, the number of bisection steps required is $O(\log(1/\varepsilon))$. Moreover, each iteration involves solving the subproblem $\max_{\bp\in\mathcal{P}} G(\bp,\theta)$, which, as shown later, admits a reformulation as an ECP and is therefore solvable in polynomial time.

% \begin{algorithm}[htb]
% \caption{Bisection Method for Constrained MNL Pricing}
% \label{alg:bisection_mnl}
% \begin{algorithmic}[1]
% \STATE \textbf{Input:} Tolerance $\varepsilon > 0$, bounds $\theta_{\min}, \theta_{\max}$
% \WHILE{$\theta_{\max} - \theta_{\min} > \varepsilon$}
%     \STATE $\theta \leftarrow (\theta_{\min} + \theta_{\max})/2$ \hfill \textcolor{blue}{\# Update midpoint}
%     \STATE Solve subproblem to get $v^* \leftarrow \max_{\mathbf{p} \in \mathcal{P}} G(\mathbf{p},\theta)$
    
%     \IF{$v^* \ge 0$}
%         \STATE $\theta_{\min} \leftarrow \theta$ \hfill \textcolor{blue}{\# Shift lower bound up}
%     \ELSE
%         \STATE $\theta_{\max} \leftarrow \theta$ \hfill \textcolor{blue}{\# Shift upper bound down}
%     \ENDIF
% \ENDWHILE
% \STATE \textbf{Output:} $\theta^* \leftarrow \theta_{\min}$
% \end{algorithmic}
% \end{algorithm}
\vspace{0.5em}

\noindent\textbf{Solving the subproblem $\max_{\bp\in\mathcal{P}} G(\bp,\theta)$.}
We now discuss how to efficiently solve the subproblem $\max_{\bp\in\mathcal{P}} G(\bp,\theta)$ for a fixed value of $\theta$. 
% Recall that this subproblem takes the form
% \begin{align*}
%      \max_{\bp} \quad & G(\bp,\theta )=   
%      \sum_{i\in[m]} p_i \exp(a_i-b_i p_i) 
%      -\theta\left(1+\sum_{j\in[m]} \exp(a_j-b_j p_j)\right) \\
%      \text{s.t. } \quad 
%      & \sum_{i=1}^m \alpha_{ki} p_i \leq \beta_k, \quad k\in [K], \\
%      & L_i \leq p_i \leq U_i, \quad \forall i \in [m].
% \end{align*}
The objective function $G(\bp,\theta)$ is neither convex nor concave in $\bp$ and can be viewed as a difference-of-convex function, which precludes direct application of standard convex optimization techniques. Our key idea is to apply a change of variables that transforms the problem into an equivalent convex optimization problem. Specifically, define
$u_i = \exp(a_i-b_i p_i),~\forall i\in[m],$
which implies
$p_i = \frac{a_i-\ln u_i}{b_i}.$
Substituting this transformation into the objective and constraints and rearranging terms, the subproblem can be written compactly as:
% yields an equivalent formulation in terms of $\bu=(u_1,\ldots,u_m)$:
% \begin{align*}
%      \max_{\bu} \quad & G(\bu,\theta )=   
%      \sum_{i\in[m]} \frac{u_i (a_i-\ln u_i)}{b_i} 
%      -\theta\left(1+\sum_{j\in[m]} u_j\right) \\
%      \text{s.t. } \quad 
%      & \sum_{i=1}^m \alpha_{ki} \frac{a_i-\ln u_i}{b_i} \leq \beta_k, \quad k\in [K], \\
%      & L_i \leq \frac{a_i-\ln u_i}{b_i} \leq U_i, \quad \forall i \in [m].
% \end{align*}
% Rearranging terms, the subproblem can be written compactly as
\begin{align}
\max_{\bu>0}\quad 
& G(\bu,\theta)
= -\theta + \sum_{i=1}^m \left( \frac{a_i}{b_i}\,u_i - \frac{u_i \ln u_i}{b_i} - \theta\,u_i \right) \label{prob:PO-MNL-Sub} \\
\text{s.t.}\quad
& \sum_{i=1}^m \frac{\alpha_{ki}}{b_i} \ln u_i \geq 
\sum_{i=1}^m \frac{\alpha_{ki}a_i}{b_i} - \beta_k, 
\qquad \forall k \in [K], \nonumber \\
& e^{a_i-b_i U_i} \le u_i \le e^{a_i-b_i L_i}, 
\qquad \forall i \in [m]. \nonumber
\end{align}
It can be seen that the objective in \eqref{prob:PO-MNL-Sub} is concave in $\bu$, since the function $-u_i\ln u_i$ is concave for $u_i>0$, and all remaining terms are linear. Moreover, the constraints are convex, as they involve affine bounds and linear constraints in $\ln u_i$. As a result, the subproblem reduces to a convex optimization problem that  can therefore be solved efficiently using convex optimization methods.

An important observation is that problem~\eqref{prob:PO-MNL-Sub} admits an equivalent reformulation as an ECP, which can be solved efficiently in polynomial time using off-the-shelf solvers such as MOSEK. We formalize this result in the following proposition.

\begin{proposition}\label{prop:MNL-ECP}
Problem~\eqref{prob:PO-MNL-Sub} is equivalent to the following ECP:
\begin{align*}
\max_{\bu,\bs,\bv}\quad 
& -\theta \;+\; \sum_{i=1}^m \Bigl(\tfrac{a_i}{b_i}\,u_i \;-\; \tfrac{1}{b_i}\,s_i \;-\; \theta\,u_i \Bigr) \\
\text{s.t.}\quad
& \sum_{i=1}^m \tfrac{\alpha_{ki}}{b_i}\, v_i 
\;\ge\; \sum_{i=1}^m \tfrac{\alpha_{ki} a_i}{b_i} \;-\; \beta_k,
\qquad \forall k\in[K], \\
& e^{a_i - b_i U_i} \;\le\; u_i \;\le\; e^{a_i - b_i L_i},
\qquad \forall i\in[m], \\
& (-s_i,\,u_i,\,1)\ \in\ \mathcal{K}_{\exp}, 
\qquad \forall i\in[m], \\
& (v_i,\,1,\,u_i)\ \in\ \mathcal{K}_{\exp},
\qquad \forall i\in[m],
\end{align*}
where the exponential cone is defined as
\[
\mathcal{K}_{\exp}
=\bigl\{(x,y,z)\in\mathbb{R}^3 : y>0,\; y e^{x/y}\le z\bigr\}
\;\cup\;
\bigl\{(x,0,z)\in\mathbb{R}^3 : x\le 0,\; z\ge 0\bigr\}.
\]
\end{proposition}
The exponential cone reformulation in Proposition \ref{prop:MNL-ECP} offers both theoretical and practical advantages. From a computational standpoint, an ECP can be solved efficiently using modern interior-point (IP) solvers such as MOSEK \citep{mosek2023}, which are specifically designed to handle exponential cones and have demonstrated strong numerical performance in large-scale applications. From a theoretical perspective, ECPs belong to the class of convex conic optimization problems and therefore admit polynomial-time solution methods. In particular, IP algorithms solve ECPs to $\varepsilon$-accuracy in time polynomial in the problem size and $\log(1/\varepsilon)$ \citep{nesterov1994,andersen2000}. As a result, each feasibility check in our bisection-based algorithm can be carried out efficiently with provable complexity guarantees, which is essential for establishing the overall PTAS.

\begin{remark}[Quasiconcavity under variable transformation]
It is well known that the revenue function of the MNL-based pricing problem is neither concave nor quasiconcave in prices \citep{hanson1996_ms}. However, we can show that, under the above variable transformation, the objective admits much stronger structural properties. In particular, under the transformation
$u_i = \exp(a_i - b_i p_i), ~\forall i \in [m],$
the revenue function becomes strictly quasiconcave in the transformed variables.
\end{remark}
We formalize this observation in the following proposition.

\begin{proposition}
\label{prop:MNL-quasiconcavity}
Let the transformed revenue function $R(\bu)$ be defined over the feasible domain $\mathcal{U}$ as
$R(\bu)
= \frac{\sum_{i=1}^m \frac{1}{b_i} u_i \bigl(a_i - \ln u_i\bigr)}{1 + \sum_{j=1}^m u_j}.$
Then $R(\bu)$ is strictly quasiconcave on $\mathcal{U}$. Consequently, the optimization problem admits a unique global maximizer, and any local maximizer of $R(\bu)$ over a convex feasible set is globally optimal.
\end{proposition}
Strict quasiconcavity implies unimodality over convex domains. Consequently, the transformed revenue function $R(\bu)$ admits a unique global maximizer, and any stationary point reached by a gradient-based method is globally optimal. This property suggests that, in principle, first-order methods could be used to solve the transformed problem directly. However, directly optimizing the quasiconcave objective poses several challenges. First, gradient-based methods generally require careful step-size selection and may exhibit slow convergence near the optimum, particularly when the objective is flat or ill-conditioned. Second, enforcing the linear constraints and bound constraints on $\bu$ within a gradient-based framework typically necessitates projection or barrier techniques, which complicate both implementation and convergence analysis. Finally, such approaches do not naturally yield explicit accuracy guarantees in terms of $\varepsilon$-optimality.

In contrast, the bisection-based approach combined with ECP leverages the monotonic structure induced by the fractional reformulation and reduces the problem to a sequence of convex feasibility checks. Each subproblem can be solved efficiently and robustly using IP methods with polynomial-time complexity and logarithmic dependence on the desired accuracy. This approach therefore provides stronger theoretical guarantees, greater numerical stability, and a clear path to establishing a PTAS.

\subsection{Complexity Analysis}
We now analyze the time complexity of solving the constrained pricing problem under the MNL model using the proposed bisection and ECP approach. The result is summarized in the following proposition.
\begin{proposition}
\label{prop:MNL-complexity}
Assume that the number of linear constraints satisfies $K = O(m)$.
For any desired accuracy $\varepsilon>0$, the bisection scheme combined with ECP returns an approximate optimal parameter $\hat{\theta}$ such that
$|\hat{\theta}-\theta^\star|\le \varepsilon,$
where $\theta^\star$ denotes the optimal value. Moreover, when each ECP subproblem is solved using a primal--dual IP method, the total arithmetic complexity of the algorithm is
${O}\!\left(m^{3.5}\,\log^2\!\left(\tfrac{1}{\varepsilon}\right)\right).$
\end{proposition}
The proof of Proposition~\ref{prop:MNL-complexity} (provided in Appendix~\ref{apd:proofs}) proceeds by exploiting the monotone structure induced by the fractional reformulation. The key observation is that the value function $\phi(\theta) = \max_{\bp\in\mathcal{P}} G(\bp,\theta)$ is strictly decreasing with a uniformly bounded slope, which ensures both the uniqueness of the optimal parameter and the robustness of the bisection procedure to inexact subproblem solutions. This allows us to decouple accuracy requirements across the outer bisection loop and the inner ECP, showing that it suffices to solve each subproblem to $\Theta(\varepsilon)$ accuracy. The total complexity then follows by combining the logarithmic convergence rate of bisection with the polynomial-time complexity of IP methods for ECP.

From an algorithmic perspective, Proposition~\ref{prop:MNL-complexity} establishes that the proposed method constitutes a PTAS for the constrained pricing problem under the MNL model. For any desired accuracy $\varepsilon>0$, the algorithm computes an $\varepsilon$-optimal solution in time polynomial in the problem size and logarithmic in $1/\varepsilon$. This result is particularly significant given the inherent nonconcavity of the pricing objective in prices, and it provides a strong theoretical guarantee that constrained pricing under MNL demand can be solved efficiently and reliably using convex optimization techniques.

\section{Constrained Pricing under FMNL Models} \label{sec:pricing-FMNL}
In this section, we extend our analysis to constrained pricing under FMNL demand models, which capture heterogeneous customer preferences through a finite number of segments. Unlike the MNL case, the resulting revenue function exhibits a sum-of-ratios structure and is highly nonconvex, leading to multiple local optima even in simple settings (as shown in Figure \ref{fig:FMNL_example}). We show how to exploit the structure of FMNL demand to reformulate the problem in a manner that isolates the nonconvexity, laying the groundwork for a globally reliable approximation algorithm.
% We first write  the constrained pricing problem under FMNL model with T customer segments:
% \begin{align*}
%      \max_{\bp} \quad & R^{\text{FMNL}}(\bp) = \sum_{t=1}^T d_t \frac{ \sum_{i\in[m]} p_i\exp(a_{ti}-b_{i}p_i)}{1+\sum_{j\in[m]} \exp(a_{tj}-b_{j}p_j)} \\
%      \text{s.t. } \quad 
%      & \sum_{i=1}^m \alpha_{ki} p_i \leq \beta_k, \quad k\in [K], \\
%      & L_i \leq p_i \leq U_i, \quad \forall i \in [m].
% \end{align*}
Our general idea is to reformulate the problem as a \emph{bilinear-convex program}—a program 
that contains some bilinear terms together with convex constraints, where the number of bilinear 
terms is proportional to the number of customer segments. As a consequence, a B\&B approach that employs McCormick relaxations, combined with a convex solver, can be 
applied to solve the problem. Since the number of bilinear terms is only $\mathcal{O}(T)$, the 
overall complexity is exponential in $T$ while remaining polynomial with respect to the other 
problem parameters.

\subsection{Bilinear-convex Reformulation}
Our first step is to reformulate the constrained pricing problem under the FMNL model. As in the MNL case, we leverage a change of variables to expose useful convex structure and isolate the nonconvex components of the problem. In particular, recall that under FMNL demand, the expected revenue is a weighted sum of segment-level revenues, each of which takes a fractional form. To facilitate reformulation, we introduce the following auxiliary variables:
\[
y_t = \sum_{i\in[m]} p_i \exp(a_{ti}-b_i p_i),\quad z_t = 1 + \sum_{j\in[m]} \exp(a_{tj}-b_j p_j), \quad \forall t\in[T],
\]
\[
u_i = \exp(-b_i p_i), \quad \forall i\in[m].
\]
Here, $y_t$ and $z_t$ represent the numerator and denominator of the revenue contribution from customer segment $t$, respectively, while the variables $u_i$ correspond to the same exponential transformation used in the MNL case and allow prices to be recovered via $p_i = -\ln u_i / b_i$. Using these variables, the constrained FMNL pricing problem in \eqref{prob:Price-MXL} can be written equivalently as:
\begin{align}
    \max_{\bu,\by,\bz}\quad & 
    \sum_{t=1}^T d_t \frac{y_t}{z_t} \label{prob:MXL-reform-u}\\
    \text{s.t.}\quad 
    & y_t \le \sum_{i=1}^m \Bigl(-\tfrac{e^{a_{ti}}}{b_i}\Bigr) u_i \ln u_i
    ;\quad z_t \ge 1 + \sum_{j=1}^m e^{a_{tj}} u_j, 
    \qquad t\in[T], \nonumber\\
    & \sum_{i=1}^m \alpha_{ki} \Bigl(-\tfrac{\ln u_i}{b_i}\Bigr) \le \beta_k, 
    \qquad k\in[K], \nonumber\\
    & L_i \le -\tfrac{\ln u_i}{b_i} \le U_i; 
    \quad u_i > 0, \qquad i\in[m]. \nonumber
\end{align}
To further isolate the nonconvexity, define $\theta_t := y_t / z_t$ for each segment $t$. As in the MNL case, the nonlinear terms involving $u_i \ln u_i$ and $\ln u_i$ can be handled using exponential cone constraints. Introducing auxiliary variables $\bs$ and $\bv$, we obtain the following bilinear exponential-cone program:
\begin{align}
    \max_{\bu,\by,\bz,\bs,\bv,\btheta}\quad & 
    \sum_{t=1}^T d_t \theta_t \label{prob:MXL-bilinear}\\
    \text{s.t.}\quad 
    & y_t = \theta_t z_t;\quad y_t \le \sum_{i=1}^m \Bigl(-\tfrac{e^{a_{ti}}}{b_i}\Bigr) s_i;\quad z_t \ge 1 + \sum_{j=1}^m e^{a_{tj}} u_j, 
    && t\in[T], \nonumber\\
    & \sum_{i=1}^m \tfrac{\alpha_{ki}}{b_i} (-v_i) \le \beta_k, 
    && k\in[K], \nonumber \\
    & -b_i U_i \le v_i \le -b_i L_i;\quad (-s_i,\,u_i,\,1)\in\mathcal{K}_{\exp}; \quad
   (v_i,\,1,\,u_i)\in\mathcal{K}_{\exp}, 
    && i\in[m],  \nonumber
\end{align}
In formulation~\eqref{prob:MXL-bilinear}, all nonlinearities except the bilinear constraints $y_t=\theta_t z_t$ are captured through convex exponential cone constraints. This reformulation therefore isolates the nonconvexity of the FMNL pricing problem into a small number of bilinear terms—one per customer segment—which will be exploited in the design of our global optimization algorithm.

\begin{remark}[Pairwise constraints on prices]
Throughout the paper, we assume that the linear price constraints take the form
$\sum_{i=1}^m \alpha_{ki} p_i \le \beta_k,~ k\in[K],$
with coefficients satisfying $\alpha_{ki}\ge 0$. This nonnegativity assumption is critical to preserve convexity after the variable transformation. Under the change of variables $u_i = e^{-b_i p_i}$, linear constraints in prices translate into constraints that are affine in $\ln u_i$. When $\alpha_{ki}\ge 0$, these constraints define a convex feasible set in the transformed space $\mathcal{U}$. In contrast, allowing coefficients of arbitrary sign generally destroys convexity of the transformed constraints, breaking the convex structure of the reformulation and making it difficult to design algorithms with global approximation guarantees. As a result, relaxing this assumption to fully general linear price constraints would substantially complicate the problem and preclude a PTAS under current techniques.

That said, many practically relevant pricing rules involve \emph{pairwise price comparisons}, such as enforcing that the price of one product does not exceed that of another by more than a fixed margin. Examples include versioning constraints (e.g., premium products priced above basic ones), fairness or regulatory rules, and internal pricing guidelines across product tiers. We now show that such constraints can be incorporated without sacrificing tractability under a mild and commonly adopted assumption.

Specifically, suppose that price sensitivity parameters are identical across products, i.e., $b_i=b$ for all $i\in[m]$. This assumption is standard in discrete choice modeling and widely used in both empirical and theoretical work, as it reflects homogeneous price responsiveness across products while allowing heterogeneity to arise through alternative-specific variables. Under this assumption, we can accommodate pairwise constraints of the form
$p_i \le p_j + r_{ij},~\text{for some } i,j\in[m],$
where $r_{ij}$ are fixed constants. Under the transformation $u_i = e^{-b p_i}$, the price can be recovered as $p_i = -\tfrac{1}{b}\ln u_i$, and the above constraint is equivalent to
 $u_j \le e^{\,b r_{ij}}\, u_i$, which is linear in $(u_i,u_j)$ and therefore preserves convexity of the feasible set in the transformed space. Consequently, such pairwise pricing constraints can be incorporated directly into the ECP without introducing additional nonconvexities or altering the bilinear--convex structure of the FMNL pricing problem. This observation shows that the modeling assumptions adopted here are flexible enough to capture a wide range of practically important pricing rules, while still enabling globally reliable approximation algorithms.
\end{remark}

\subsection{A B\&B Scheme for the Bilinear--Convex Program}

We now describe our B\&B approach. As shown later, the resulting procedure runs in time polynomial in the problem size (e.g., the number of products and constraints) and exponential only in the number of customer segments $T$. The key observation is that the formulation contains a very limited source of non-convexity: the $T$ bilinear equalities $y_t=\theta_t z_t$, one for each customer segment. All remaining components of the model are convex and admit an ECP.

The B\&B scheme exploits this structure by branching exclusively on the bilinear terms. At each node of the search tree, we solve a \emph{convex} exponential-cone relaxation obtained by replacing each bilinear equality $y_t=\theta_t z_t$ with its McCormick envelope \citep{mccormick1976}  over the current bounds of $(\theta_t,z_t)$. As the algorithm explores the tree, these bounds are progressively tightened and propagated, leading to increasingly strong relaxations and enabling effective pruning. This structural decomposition ensures that the exponential complexity is confined to the number of customer segments, while scalability with respect to the number of products is preserved.
\subsubsection{Node Relaxation (ECP \& McCormick Envelope)}\label{sec:node_relaxation}
Let $[\underline{\theta}_t,\overline{\theta}_t]$ and $[\underline{z}_t,\overline{z}_t]$
denote valid bounds on $\theta_t$ and $z_t$, respectively, at a given node of the B\&B tree. The bilinear identity $y_t=\theta_t z_t$ is relaxed using the standard McCormick envelope \citep{mccormick1976}, yielding the following four linear inequalities:
\begin{align}
y_t &\;\ge\; \underline{\theta}_t\, z_t \;+\; \underline{z}_t\, \theta_t \;-\; \underline{\theta}_t\,\underline{z}_t; \quad y_t \;\ge\; \overline{\theta}_t\, z_t \;+\; \overline{z}_t\, \theta_t \;-\; \overline{\theta}_t\,\overline{z}_t, \label{mc1}\\
y_t &\;\le\; \overline{\theta}_t\, z_t \;+\; \underline{z}_t\, \theta_t \;-\; \overline{\theta}_t\,\underline{z}_t;\quad y_t \;\le\; \underline{\theta}_t\, z_t \;+\; \overline{z}_t\, \theta_t \;-\; \underline{\theta}_t\,\overline{z}_t. \label{mc4}
\end{align}
These inequalities define the tightest possible convex relaxation of the bilinear term over the given bounding box. Importantly, the relaxation error depends directly on the widths of the intervals $[\underline{\theta}_t,\overline{\theta}_t]$ and $[\underline{z}_t,\overline{z}_t]$: as these bounds are tightened through branching and bound propagation, the McCormick envelope converges uniformly to the exact bilinear constraint. In the limit where either interval collapses to a singleton, the relaxation becomes exact. At each node, replacing $y_t=\theta_t z_t$ with \eqref{mc1}--\eqref{mc4} yields a \emph{convex} ECP. Solving this relaxation provides an upper bound on the optimal objective value of the original nonconvex problem at the node, which is used to guide branching and pruning in the B\&B procedure.
\subsubsection{Global Variable Bounds and Initialization}
We initialize the bounds for the auxiliary variables using interval arithmetic based on the bounds of the price variables $p_i \in [L_i, U_i]$. First, recall that the auxiliary variable $u_i$ is defined as $u_i = e^{-b_i p_i}$, resulting in the following bounds:
$\underline{u}_i \;=\; e^{-b_i U_i}$ and $
\overline{u}_i \;=\; e^{-b_i L_i},\quad \forall i\in[m].$
The denominator variable $z_t$ is defined as $z_t = 1 + \sum_{i=1}^m e^{a_{ti}} u_i$. Since the coefficients $e^{a_{ti}}$ are positive, $z_t$ is strictly increasing in $u_i$, and its initial bounds are given by
$\underline{z}_t \;=\; 1+\sum_{i=1}^m e^{a_{ti}} \,\underline{u}_i,$  and $
\overline{z}_t \;=\; 1+\sum_{i=1}^m e^{a_{ti}} \,\overline{u}_i,\quad \forall t\in[T].$
To derive bounds for $\theta_t$, we first analyze the numerator term $y_t$. Let $s_i = u_i \ln u_i$. The function $f(u) = u \ln u$ is convex with a global minimum at $u = 1/e$. Therefore, we compute the bounds $[\underline{s}_i, \overline{s}_i]$ by checking the critical point against the domain $[\underline{u}_i, \overline{u}_i]$:
\[
\underline{s}_i = 
\begin{cases} 
-e^{-1} & \text{if } \underline{u}_i \le e^{-1} \le \overline{u}_i \\
\min(\underline{u}_i \ln \underline{u}_i, \; \overline{u}_i \ln \overline{u}_i) & \text{otherwise}
\end{cases},
\qquad
\overline{s}_i = \max(\underline{u}_i \ln \underline{u}_i, \; \overline{u}_i \ln \overline{u}_i).
\]
The numerator $y_t$ is a linear combination of these terms: $y_t = \sum_{i=1}^m \left(-\frac{e^{a_{ti}}}{b_i}\right) s_i$. Since the coefficients $C_{ti} = -e^{a_{ti}}/b_i$ are strictly negative, the lower bound of $y_t$ is determining by the \textit{upper} bound of $s_i$, and vice versa:
$\underline{y}_t = \sum_{i=1}^m C_{ti} \,\overline{s}_i$ and $
\overline{y}_t = \sum_{i=1}^m C_{ti} \,\underline{s}_i.$
Finally, the bounds for $\theta_t$ are obtained via interval division. Since $z_t > 0$ and $y_t$ is negative, the division operations are:
$\underline{\theta}_t \;=\; \frac{\underline{y}_t}{\underline{z}_t}$ and $
\overline{\theta}_t \;=\; \frac{\overline{y}_t}{\overline{z}_t}.$
\subsubsection{Branching, Bounding and Pruning}
Our proposed B\&B framework for the constrained pricing problem is outlined in Algorithm~\ref{alg:bb_compact}. We employ a best-first search (BFS) policy, maintaining a priority queue $\mathcal{Q}$ ordered by the nodes' objective value. To ensure computational efficiency, we utilize a persistent solver architecture: rather than rebuilding the optimization model $\mathcal{M}$ at each iteration, we dynamically update the variable bounds in the existing model and re-solve to obtain the local upper bound $U_N$ and relaxation candidates $(\hat{\bu}, \hat{\by}, \hat{\bz}, \hat{\btheta})$. The algorithm proceeds through three critical phases at each node:
\begin{enumerate}
    \item \textit{Pruning. } Nodes are immediately discarded if the relaxation problem $\mathcal{M}$ is proven infeasible or if the local upper bound $U_N$ is less than or equal to the current best global lower bound $R^*$. This effectively prunes sub-optimal regions of the search space.
    \item \textit{Incumbent update \& optimality check. } If the node is not pruned, we attempt to update the global lower bound $R^*$ and best price vector $\bp^*$ using the candidate price $\hat{\bp}$. We then evaluate the bilinear feasibility by calculating the maximum violation $\delta_{\max} = \max_t |\hat{y}_t - \hat{\theta}_t \hat{z}_t|$. Under the BFS, if this violation is within the tolerance $\varepsilon$, the current node is guaranteed to be the global optimum, and the search terminates immediately.
    \item \textit{Branching. } If the solution remains infeasible w.r.t. the bilinear constraints ($\delta_{\max} > \varepsilon$), we perform spatial branching. We identify the index $t^*$ with the maximum violation and select the variable involved in the bilinear term (either $\theta_{t^*}$ or $z_{t^*}$) that has the largest current interval. The search space is then partitioned by bisecting this interval, and the resulting child nodes are added to $\mathcal{Q}$. 
\end{enumerate}
\begin{algorithm}[htb]
\caption{B\&B for Constrained Pricing Problem under FMNL}
\label{alg:bb_compact}
\begin{algorithmic}[1]
\STATE \textbf{Input:} $\mathcal{Q} \leftarrow \{\text{root}\}$, $R^* \leftarrow -\infty$, $p^* \leftarrow \emptyset$, $\mathcal{M} \leftarrow \{\text{global bounds of } u, z, \theta \text{ and static constraints}\}$
\WHILE{$\mathcal{Q} \neq \emptyset$ \AND time limit not reached}
    \STATE Node $N \leftarrow \text{pop}(\mathcal{Q})$ \hfill \textcolor{blue}{\# Select the best node}
    \STATE Update parameters' bounds in $\mathcal{M}$ using at Node $N$ \hfill \textcolor{blue}{\# Persistent solver update}
    \STATE Solve $\mathcal{M}$ (using ECP solver, i.e., Mosek) to obtain upper bound $U_N$ and candidate $(\hat{u}, \hat{y}, \hat{z}, \hat{\theta})$
    
    \STATE \textbf{if} $\mathcal{M}$ infeasible \textbf{or} $U_N \le R^*$ \textbf{then continue} \hfill \textcolor{blue}{\# Prune by infeasibility or bound}
    
    \STATE Update $R^*, p^*$ using $\hat{p}=-\frac{\ln(\hat{u})}{b}$ if $R(\hat{p}) > R^*$ \hfill \textcolor{blue}{\# Update incumbent}
    \STATE Let $\delta_{\max} \leftarrow \max_t |\hat{y}_t - \hat{\theta}_t \hat{z}_t|$ \hfill \textcolor{blue}{\# Calculate max bilinear violation}
    
    \IF{$\delta_{\max} \le \varepsilon$}
        \STATE \textbf{break} \hfill \textcolor{blue}{\# Global optimality reached}
    \ENDIF

    \STATE $t^* \leftarrow \arg\max \delta_t$
    \STATE Select branch variable $v \in \{\theta_{t^*}, z_{t^*}\}$ with largest interval
    \STATE Branch on $v$: bisect interval, add children to $\mathcal{Q}$ \hfill \textcolor{blue}{\# Spatial branching}
\ENDWHILE
\STATE \textbf{Output:} $p^*, R^*$
\end{algorithmic}
\end{algorithm}
Our strategy of branching solely on the maximum violation is driven by three key factors: priority, efficiency, and completeness. First, regarding priority, the term with the largest violation $\delta_{\max}$ acts as the primary driver of the relaxation's overestimation. Targeting this error yields the greatest reduction in the local upper bound $U_N$, thereby maximizing the chance of pruning the node against the global lower bound $R^*$. Second, this approach ensures efficiency by maintaining a binary tree structure; attempting to branch on multiple violations simultaneously would lead to a combinatorial explosion of child nodes. Finally, the method guarantees completeness. Secondary violations are not ignored but deferred; if they remain significant after the primary violation is resolved, they will naturally surface as the maximum violation in descendant nodes, ensuring all non-convexities are eventually corrected.
Further details of the B\&B algorithm are provided in Appendix~\ref{apd:detailed-BnB}.

\subsection{Complexity Discussion}
We now discuss the computational complexity of the proposed B\&B procedure. The overall complexity consists of two components: the cost of solving the convex ECP at each B\&B node, and an upper bound on the number of nodes generated in the search tree.

At each node, the algorithm solves a convex ECP whose barrier parameter satisfies $\nu=\Theta(m)$, corresponding to two exponential cone constraints per product and additional linear constraints. A primal--dual IP method requires $\mathcal{O}(\sqrt{m}\log(1/\delta))$ Newton iterations to achieve accuracy $\delta$, and each iteration involves solving a linear system of dimension $n=\Theta(m+K)$, with worst-case arithmetic cost ${\mathcal{O}}(n^3)$ \citep{nesterov1994}. Thus, the computational cost per node is polynomial in the problem dimensions.

On the other hand, the number of B\&B nodes depends on the number of bilinear terms in the formulation. Since the only source of nonconvexity consists of the $T$ bilinear constraints $y_t=\theta_t z_t$, the size of the B\&B tree grows exponentially in $T$ in the worst case. Importantly, this exponential dependence is independent of the number of products. In practice, strong relaxations, effective branching on the most violated bilinear terms, and warm-starting IP solvers substantially reduce the number of explored nodes. We summarize the worst-case complexity guarantee in the following theorem.
\begin{theorem}[Complexity of B\&B for the constrained pricing under the FMNL model]
\label{prop:bnb_complexity}
For any optimality tolerance $\varepsilon \in (0,1)$, the spatial B\&B algorithm terminates with an $\varepsilon$-optimal solution. For a fixed number of customer segments $T$, the total time complexity is bounded by
\[
\underbrace{\mathcal{O}\!\left(\left(\tfrac{1}{\varepsilon}\right)^T\right)}_{\text{number of nodes}}
\;\times\;
\underbrace{\mathcal{O}\!\left(n^3 \sqrt{m}\log\!\left(\tfrac{1}{\varepsilon}\right)\right)}_{\text{ECP solver cost}},
\]
where $n=\Theta(m+K)$ denotes the problem dimension at each node. In particular, if the number of linear constraints satisfies $K=\mathcal{O}(m)$ and $T$ is bounded, the overall complexity simplifies to
$\mathcal{O}\!\left(\left(\tfrac{1}{\varepsilon}\right)^T m^{3.5}\log\!\left(\tfrac{1}{\varepsilon}\right)\right).$
Consequently, when the number of customer segments $T$ is bounded, the proposed algorithm runs in time polynomial in the problem size and the inverse accuracy $1/\varepsilon$.
\end{theorem}
Theorem~\ref{prop:bnb_complexity} establishes that the proposed B\&B algorithm computes a globally $\varepsilon$-optimal solution for the constrained pricing problem under FMNL demand. Importantly, the exponential dependence in the worst-case complexity is confined solely to the number of customer segments $T$, while the dependence on the number of products and constraints remains polynomial. This distinction is critical: in practical applications, customer heterogeneity is typically captured using a small number of latent segments, both for interpretability and for statistical reliability in demand estimation. As a result, $T$ is usually modest, even in large-scale pricing problems with extensive assortments. Under this commonly encountered regime, the proposed algorithm scales efficiently with the number of products and constraints, making it well suited for realistic pricing applications. Moreover, the explicit dependence on the inverse accuracy $1/\varepsilon$ provides a transparent trade-off between solution precision and computational effort, allowing practitioners to adjust accuracy levels according to operational requirements.

\begin{remark}[Practical complexity]
While Proposition~\ref{prop:bnb_complexity} establishes a worst-case complexity bound of 
$\mathcal{O}\!\left((1/\varepsilon)^T\, m^{3.5}\log(1/\varepsilon)\right)$
-- corresponding to a scenario where the B\&B algorithm essentially resorts to a full grid search--this estimate is highly pessimistic. In practice, the bounding mechanism significantly prunes the search space. If the convex relaxation at a node provides a tight bound, or if the current incumbent solution is strong, sub-optimal branches are ``fathomed'' (discarded) early in the tree traversal. In favorable scenarios where one child node is consistently pruned at each branching step, the algorithm avoids enumerating all leaf nodes, and the search effectively follows a limited set of paths to the optimum. Under such conditions, the number of explored nodes scales proportionally to the tree depth rather than the tree size, yielding an effective complexity closer to $O((\log(1/\epsilon))^T)$. This behavior aligns with our experimental results, where the algorithm demonstrates stable scaling even with small precision thresholds $\epsilon$. 
\end{remark}

\begin{remark}[Comparison to the bisection--ECP scheme for MNL-based pricing]
The B\&B framework developed for the FMNL setting can, in principle, also be applied to the MNL-based pricing problem discussed in Section~3. In this case, the formulation contains a single bilinear term ($T=1$), and the resulting B\&B procedure yields a worst-case time complexity of
$\mathcal{O}\!\left((1/\varepsilon)\, m^{3.5}\log(1/\varepsilon)\right).$
While this complexity is polynomial, its dependence on the inverse accuracy $1/\varepsilon$ is linear. In contrast, the bisection--ECP scheme developed specifically for the MNL case exploits the unimodal structure of the objective and achieves a strictly faster convergence rate. Its overall complexity is
$\mathcal{O}\!\left(m^{3.5}\log^2(1/\varepsilon)\right),$
which depends only logarithmically on the target accuracy. This difference highlights the advantage of leveraging problem-specific structure: although B\&B provides a unified framework applicable to both MNL and FMNL models, it does not fully exploit the special properties of the single-segment case. As a result, the bisection--ECP approach is provably more efficient for MNL-based pricing and is therefore preferable whenever the unimodal structure can be used.
\end{remark}
\begin{remark}[Comparison to the B\&B complexity in \citep{ruben2022mnsc}]
\cite{ruben2022mnsc} develop a B\&B algorithm for the \emph{unconstrained} pricing problem under FMNL demand and establish a worst-case time complexity of
$\mathcal{O}\!\left((1/\varepsilon)^T\, m^{5.5+3T}\right)$.
In contrast, Theorem~\ref{prop:bnb_complexity} shows that the complexity of our B\&B algorithm for the \emph{constrained} pricing problem scales as
$\mathcal{O}\!\left((1/\varepsilon)^T\, m^{3.5}\log(1/\varepsilon)\right)$ when $K=\mathcal{O}(m)$.

Two key differences are worth highlighting. First, in our formulation the exponential dependence is confined solely to the number of customer segments $T$, while the polynomial dependence on the number of products is independent of $T$. By contrast, the complexity bound in \citep{ruben2022mnsc} grows polynomially in $m$ with an exponent that itself increases with $T$. Second, despite addressing the more general setting with linear price constraints, our approach achieves a \textit{strictly lower polynomial dependence} on the problem dimension. This improvement stems from the ECP, which isolates all nonconvexities into $T$ bilinear terms and enables tighter convex relaxations within the B\&B framework.
\end{remark}
\section{Experimental Results}\label{sec:experiments}
In this section, we first describe the instance generation procedure, the benchmark methods considered—including the B\&B algorithm for unconstrained FMNL pricing from \cite{ruben2022mnsc} and a gradient-based local search heuristic—and the computational environment. We then evaluate the performance of our proposed methods relative to these benchmarks under unconstrained pricing, capacity-constrained pricing, and joint settings with pairwise price and capacity constraints.
\subsection{Data Generation \& Experimental Setting}
\paragraph{Instance Generation. }The unconstrained instances are generated following the procedure in \cite{ruben2022mnsc}. The number of products $m$ varies from 10 to 100, and the number of customer segments $T$ is chosen from $\{1,2,3,4\}$. The customer choice probability $a$ and the price sensitivity $b$ are uniformly generated from $[-7, 7]$ and $[0.001,0.01]$, respectively. For the constrained pricing dataset, the upper bound $U_i$ and the lower bound $L_i$ of the price variable $p_i~(i \in [m])$ are randomly generated in $[0.8UB_i,UB_i]$ and $[\min\{LB_i,0.1U_i\},0.3U_i]$, where: $UB_i$ and $LB_i$ are the theoretical upper bound and lower bound of $p_i$ calculated based on $a$ and $b$ (see \cite{ruben2022mnsc} for details). There are 5 capacity constraints $(\sum_{i \in [m]}\alpha_{ki} p_i \leq \beta_k)$ examined in each instance in which $\alpha_{ki} \in [0,1]$ such that $\sum_{i \in [m] \alpha_{ki}} = 1$ and $
\beta_k \in [\sum_{i \in [m]}\alpha_{ki}L_i, 0.5\bigl(\sum_{i \in [m]}\alpha_{ki}U_i\bigr)]$. To ensure the feasibility, the pairwise-price parameters $r_{ij}~(i,j\in[m],i\neq j)$ are generated as $r_{ij} = \gamma U_i$, where $\gamma$ is randomly chosen in $[0.2,0.5]$.

\paragraph{Baselines. }For unconstrained pricing, we compare our methods against the B\&B approach proposed by \cite{ruben2022mnsc}, denoted by ``Ucon-B\&B'', which is currently the state-of-the-art approach for unconstrained  price optimization under MNL and FMNL models. For constrained problems lacking exact benchmarks, we utilize a gradient-based baseline. In selecting this baseline, we evaluated two distinct optimization strategies within the \texttt{SciPy} library \citep{SciPy2020}: \texttt{SLSQP} and \texttt{trust-constr}. \texttt{SLSQP} is a sequential least squares programming algorithm that solves a series of quadratic subproblems to efficiently handle local linear approximations of constraints. In contrast, \texttt{trust-constr} employs a trust-region framework, which defines a region around the current iterate where a simplified model is trusted to be accurate, adjusting the step size dynamically to ensure convergence. Our preliminary tests revealed that \texttt{SLSQP} offers superior solution quality and runtime efficiency for our specific problem structure. Therefore, all subsequent gradient-based results refer to the \texttt{SLSQP} implementation.

\paragraph{Computational Environment. }All the experiments are implemented using Python and conducted on a PC equipped with an Intel(R) Core(TM) i7-1255U CPU @ 1.70GHz and 16 GB of RAM and the Windows 10 operating system. The exponential-cone formulations are carried out by \texttt{MOSEK} Optimization Tools version 11 with default settings, and the \texttt{SLSQP} is in \texttt{SciPy} library version 1.17.0. The optimal tolerance $\epsilon$ of all methods is set to 0.01, as in \cite{ruben2022mnsc}. The CPU time limit for each constrained pricing instance is set to 3600 seconds, i.e., the algorithms/solver are forced to stop if they exceed the time budget and report the best-found solutions.
\subsection{Constrained Pricing under FMNL: B\&B vs.\ Local Search Heuristics}
We report numerical comparisons for constrained pricing under the FMNL model; experimental results for pricing under the MNL model are deferred to Appendix~\ref{apd:experiment-MNL}. In this subsection, we focus on instances that include both capacity and pairwise price constraints, referred to as the \emph{C\&P-FMNL} setting. This setting is particularly challenging, as it combines demand heterogeneity with practical pricing constraints that couple prices across products and limit feasible price levels. Additional numerical results—including comparisons with state-of-the-art B\&B methods for unconstrained FMNL pricing (U-FMNL) \citep{ruben2022mnsc} and experiments with capacity-only constraints (C-FMNL)—are reported in Appendix~\ref{apd:FMNL-C-constraints}.

\begin{figure}[htb] \centering \includegraphics[width=\linewidth]{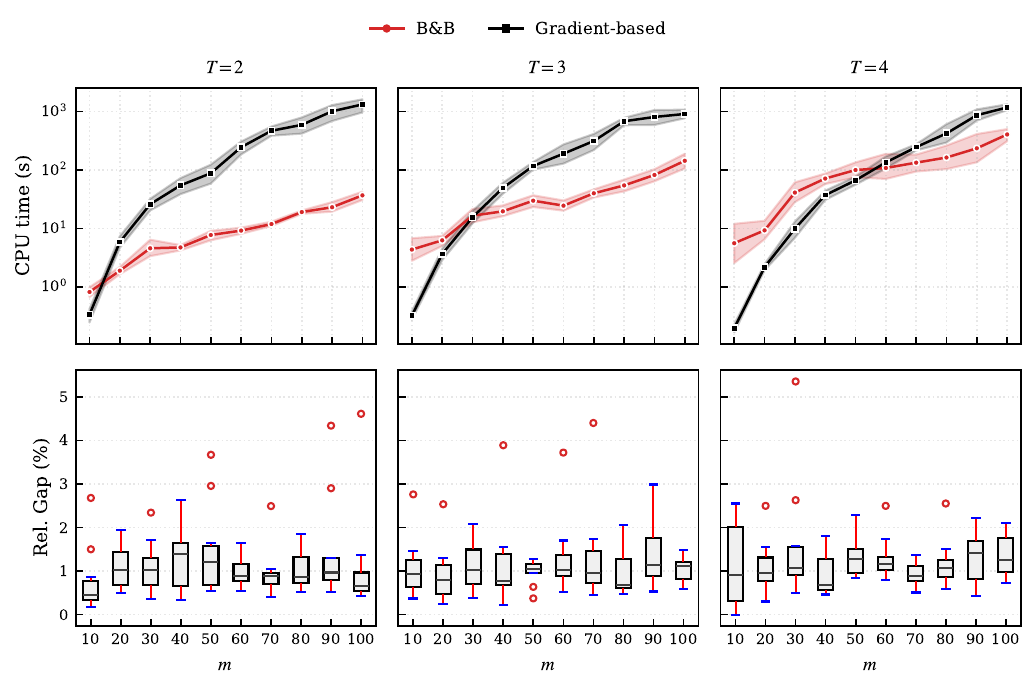} \caption{Our proposed methods vs. the gradient-based algorithm on the C\&P-FMNL dataset.} \label{fig:obj_time_FMNL_Homo} \end{figure}

Figure~\ref{fig:obj_time_FMNL_Homo} visualizes the performances of the proposed B\&B method and the gradient-based local search heuristic on the C\&P-FMNL dataset, with line plots representing a 95\% confidence interval. This setting represents the most complex scenario considered, as the presence of pairwise price constraints substantially increases problem difficulty. As the number of products $m$ grows, these constraints lead to a sharp increase in computational burden, but the two methods exhibit markedly different scalability patterns.

Specifically, the runtime of the gradient-based method grows rapidly with $m$ and appears exponential in practice, exceeding 1,000 seconds when $m$ approaches 100. In addition, despite the increased runtime, the local search heuristic fails to close the optimality gap, which remains as high as 5\% in larger instances. In contrast, the proposed B\&B approach exhibits a much more stable, near-linear increase in runtime as $m$ grows, even in the most challenging case with $T=4$ customer segments. Although the inclusion of capacity and pairwise constraints increases computational effort relative to unconstrained settings, the B\&B method consistently returns globally optimal solutions within reasonable runtimes.

These results highlight an important trade-off between speed and reliability. While gradient-based heuristics may provide quick approximations in simple or lightly constrained environments, their performance deteriorates rapidly as realistic pricing constraints and larger assortments are introduced, leading to significant and persistent revenue loss. By contrast, the proposed B\&B framework remains robust in the presence of practical constraints—such as capacity limits, price bounds, and relative pricing rules—while maintaining predictable computational scaling. This suggests that for high-stakes pricing decisions involving complex constraints, globally reliable optimization methods can deliver substantial value over heuristic approaches, even at moderate additional computational cost.
\subsection{Single-segment vs.\ Multi-segment Pricing}
\begin{figure}[htb]
    \centering
    \includegraphics[width=\linewidth]{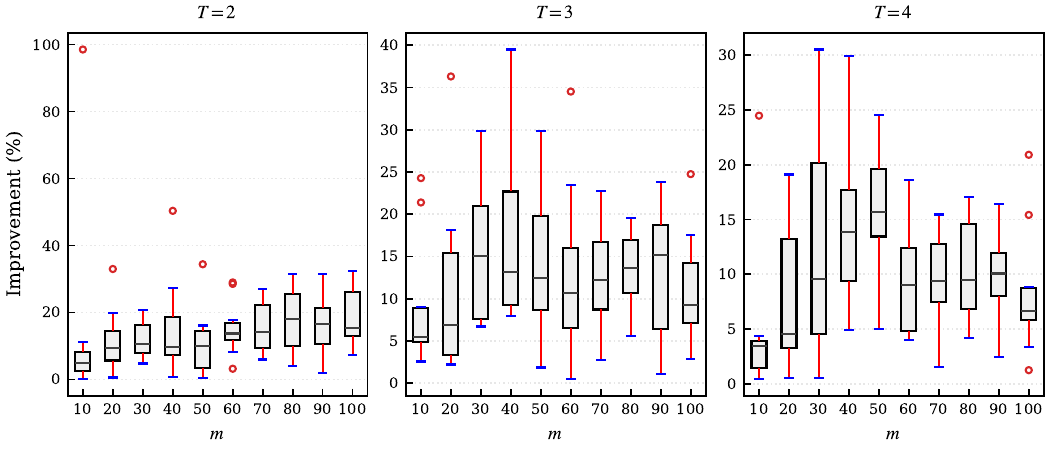}
\caption{Objective loss incurred by approximating multi-segment pricing problems with a single-segment (MNL) model.}
    \label{fig:MNL_FMNL}
\end{figure}
When $T=1$, the pricing objective reduces to a single fractional function that is unimodal in the transformed variables, which guarantees efficient convergence to the global optimum via the bisection-based method. This favorable structure naturally raises the question of whether a multi-segment pricing problem can be approximated by a single-segment model in order to exploit this tractability. To examine this possibility, we consider a common approximation strategy in practice: collapsing multiple customer segments into a single representative segment by aggregating segment-level parameters using their mean values. We then solve the resulting single-segment pricing problem and evaluate its performance against the true FMNL benchmark. Figure~\ref{fig:MNL_FMNL} reports the results of this comparison on the C\&P-FMNL dataset.

The results demonstrate that this approximation performs poorly and fails to capture the structural complexity of the multi-segment pricing problem. Across all product sizes $m$ and customer segment configurations with $T\ge2$, the single-segment surrogate consistently incurs substantial objective loss relative to the true FMNL optimum. In the most extreme cases, the approximation error reaches up to 98.5\%, and even on average the loss remains significant, typically ranging between 4\% and 18\% based on interquartile statistics. Importantly, this performance gap does not diminish as the problem size grows, indicating that the error is structural rather than numerical.

These findings highlight a critical limitation of simplifying heterogeneous demand into a single representative segment. While such approximations may appear attractive due to their computational simplicity, they can lead to systematically distorted pricing decisions when customer segments respond differently to prices and when practical constraints—such as capacity limits and relative pricing rules—are present. In such environments, ignoring demand heterogeneity can result in substantial and persistent revenue losses. The results therefore underscore the importance of explicitly accounting for multiple customer segments in constrained pricing problems, and motivate the need for optimization methods that can handle heterogeneous demand without sacrificing global optimality guarantees.
\subsection{Unconstrained Pricing as a Poor Proxy for Constrained Optima}
\begin{figure}[!h]
    \centering
    \includegraphics[width=\linewidth]{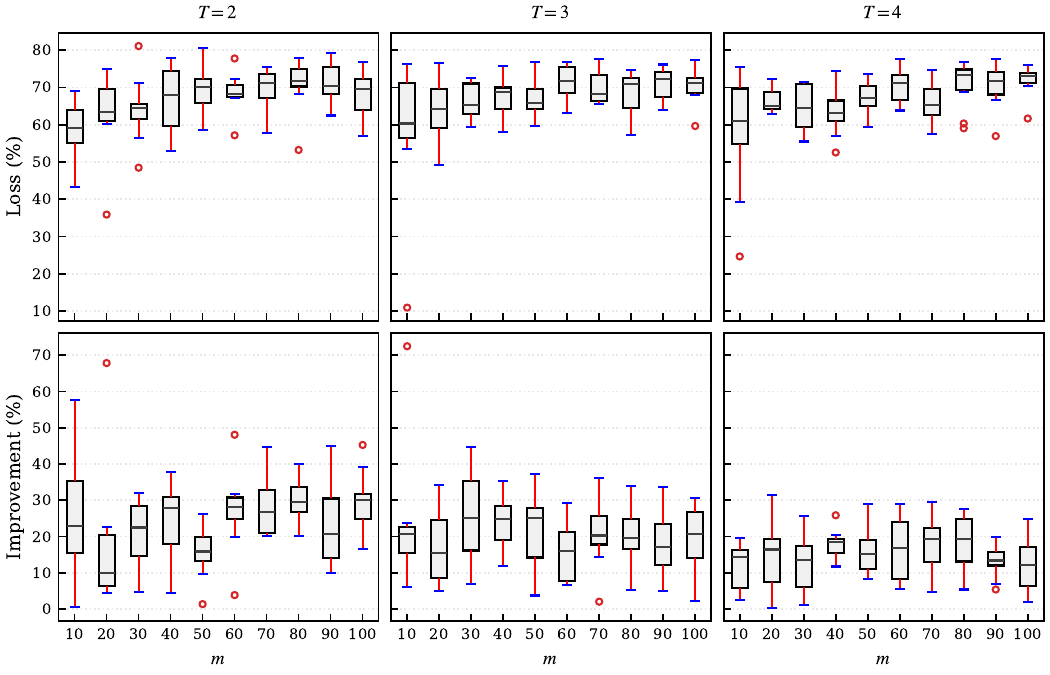}
\caption{Objective gap between projected unconstrained solutions and our B\&B method.}
    \label{fig:unconstr_constr}
\end{figure}
We examine whether unconstrained pricing solutions can serve as reasonable proxies for constrained optima in the presence of practical pricing restrictions. Figure~\ref{fig:unconstr_constr} highlights the substantial performance gap between unconstrained pricing solutions and those obtained by directly solving the constrained pricing problem under the C\&P-FMNL setting. The top row of Figure~\ref{fig:unconstr_constr} reports the objective loss incurred when an unconstrained solution is naively projected onto the feasible region defined by capacity, pairwise, and bound constraints. Specifically, given an unconstrained optimal price vector $\bp_{\text{unctr}}$, we compute a feasible price vector by solving the following projection problem using the \texttt{SLSQP} solver in the \texttt{SciPy} library:
$\min_{\bp \in \mathcal{P}} \ \|\bp - \bp_{\text{unctr}}\|_2^2,$
where $\mathcal{P}$ denotes the feasible set of price vectors $\mathbf{p}$ defined by bound, capacity, and pairwise constraints.

Across all test instances, the projected solutions exhibit severe revenue degradation, with an average objective loss of approximately 67\% and worst-case losses reaching as high as 81\%. These results demonstrate that pricing constraints fundamentally alter the structure of the optimal solution, and that ignoring them during optimization cannot be remedied by post hoc projection. Consequently, unconstrained solutions—even when adjusted to satisfy feasibility—provide poor approximations to the true constrained optimum.

The bottom row of Figure~\ref{fig:unconstr_constr} illustrates the effectiveness of the proposed B\&B approach in addressing this challenge. By directly optimizing over the constrained feasible region, the B\&B algorithm substantially mitigates the revenue loss observed under the projection baseline. Relative to the projected unconstrained solution, the B\&B method achieves average revenue improvements ranging from 15.2\% (for $T=4$) to 24.2\% (for $T=2$), with some instances recovering over 72\% of the lost objective value.

From a managerial perspective, these findings carry a clear implication: constraints such as capacity limits, price bounds, and relative pricing rules are not mere implementation details, but core determinants of optimal pricing decisions. Treating constrained pricing as an unconstrained problem followed by feasibility adjustment can lead to dramatic and systematic revenue losses. In contrast, incorporating constraints directly into the optimization process enables firms to fully capture the value of their pricing flexibility, particularly in settings with heterogeneous customer demand and complex operational restrictions.

\section{Conclusion}\label{sec:concl}
This paper studies constrained price optimization under MNL and FMNL demand models. We develop a unified optimization framework based on exponential-cone reformulations that yields global approximation guarantees despite the nonconvexity of pricing problems under logit-based demand. For the MNL model, we establish a PTAS via a convex reformulation and bisection. For FMNL models with a bounded number of customer segments, we introduce a bilinear--convex reformulation and a tailored B\&B algorithm whose complexity is exponential only in the number of segments and polynomial in all other problem dimensions. To the best of our knowledge, this is the first PTAS for constrained pricing under FMNL demand.

Our results highlight that ignoring pricing constraints or relying on local heuristics can lead to significant revenue loss under heterogeneous demand, while expressive demand models remain tractable when segmentation is kept parsimonious. Promising directions for future work include pricing under continuous-mixture or nonparametric mixed logit models, where customer heterogeneity is more richly represented but computational challenges remain.

% In the interest of anonymization, please do not include acknowledgements in your submission.
%
%\begin{acks}
%
%	The authors would like to thank Dr. Maura Turolla of Telecom
%	Italia for providing specifications about the application scenario.
%
%	The work is supported by the \grantsponsor{GS501100001809}{National
%		Natural Science Foundation of
%		China}{http://dx.doi.org/10.13039/501100001809} under Grant
%	No.:~\grantnum{GS501100001809}{61273304\_a}
%	and~\grantnum[http://www.nnsf.cn/youngscientsts]{GS501100001809}{Young
%		Scientsts' Support Program}.
%
%
%\end{acks}

%\clearpage

% Bibliography
\bibliographystyle{ACM-Reference-Format}
\bibliography{ref}

\clearpage
% Appendix
\appendix

\begin{center}\huge
    \textsc{Appendix}
\end{center}
\appendix
This appendix provides additional technical details and numerical results that complement the main text. In particular, we include omitted proofs, extended algorithmic discussions, and supplementary computational experiments that further illustrate the performance of the proposed methods. These materials are intended to support the theoretical and empirical claims made in the main paper. The appendix is organized as follows:
\begin{itemize}
  \item \textbf{Appendix \ref{apd:proofs}: Missing Proofs.} 
  Contains proofs omitted from the main text, including proofs of Propositions~\ref{prop:MNL-ECP}--\ref{prop:MNL-complexity} and Theorem~\ref{prop:bnb_complexity}.
  \item \textbf{Appendix \ref{apd:detailed-BnB}: Details of the B\&B Algorithm.} 
  Provides a detailed description of the proposed branch-and-bound algorithm for constrained pricing under FMNL demand, including the node relaxation, branching and pruning strategies, and implementation details.
  \item \textbf{Appendix \ref{apd:additional-experiments}: Additional Experiments.} 
  Presents supplementary numerical results, including comparisons under MNL and FMNL models, unconstrained and constrained settings, and detailed runtime and node-exploration analyses.
\end{itemize}

\section{Missing Proofs}\label{apd:proofs}
This section contains proofs omitted from the main paper.
\subsection{Proof of Proposition \ref{prop:MNL-ECP}}
\begin{proof}
We show that the exponential cone constraints exactly reproduce the nonlinear terms in problem~\eqref{prob:PO-MNL-Sub}.
First, consider the constraint $(-s_i,u_i,1)\in\mathcal{K}_{\exp}$. Since $u_i>0$, this is equivalent to
\[
u_i e^{-s_i/u_i} \le 1
\quad\Longleftrightarrow\quad
s_i \ge u_i \ln u_i.
\]
In the objective, $s_i$ appears with a negative coefficient $-1/b_i<0$. Therefore, at optimality $s_i$ is minimized subject to feasibility, implying
$s_i = u_i \ln u_i.$
Substituting this equality into the objective recovers the term $-\tfrac{1}{b_i}u_i\ln u_i$ exactly.

Next, consider the constraint $(v_i,1,u_i)\in\mathcal{K}_{\exp}$, which is equivalent to
\[
e^{v_i} \le u_i
\quad\Longleftrightarrow\quad
v_i \le \ln u_i.
\]
The variables $v_i$ appear only in the linear constraints
\[
\sum_{i=1}^m \tfrac{\alpha_{ki}}{b_i} v_i 
\;\ge\; \sum_{i=1}^m \tfrac{\alpha_{ki} a_i}{b_i} - \beta_k.
\]
Since increasing $v_i$ relaxes these constraints and does not worsen the objective, it follows that at optimality
$v_i = \ln u_i.$
Thus, the exponential cone constraints enforce the logarithmic terms in problem~\eqref{prob:PO-MNL-Sub} exactly. Combining the above arguments, the exponential cone program is equivalent to problem~\eqref{prob:PO-MNL-Sub}.
\end{proof}

\subsection{Proof of Proposition \ref{prop:MNL-quasiconcavity}}
\begin{proof}
We establish the result by analyzing the numerator and denominator separately and then invoking standard results on fractional quasiconcave functions. We first write $R(\bu) = N(\bu)/D(\bu)$, where
\begin{itemize}
    \item the numerator is
    \[
    N(\bu) = \sum_{i=1}^m f_i(u_i), 
    \qquad 
    f_i(u_i) := \frac{1}{b_i}\bigl(a_i u_i - u_i \ln u_i\bigr),
    \]
    \item the denominator is
    \[
    D(\bu) = 1 + \sum_{j=1}^m u_j .
    \]
\end{itemize}
Now we consider the function $g(x) = -x \ln x$ for $x>0$, which satisfies
$g''(x) = -\frac{1}{x} < 0 .$
Hence, $g$ is strictly concave on $\mathbb{R}_{++}$. Each $f_i(u_i)$ is the sum of a linear term $\frac{a_i}{b_i}u_i$ and a strictly concave term $-\frac{1}{b_i}u_i\ln u_i$, and is therefore strictly concave. Since $N(\bu)$ is a sum of strictly concave functions, it is strictly concave on $\mathcal{U}$. Moreover,  the denominator $D(\bu)$ is affine and strictly positive on $\mathcal{U}$, since $u_i>0$ for all feasible $\bu$.

Now, for any $\alpha \in \mathbb{R}$, consider the superlevel set
\[
S_\alpha := \{\bu \in \mathcal{U} \mid R(\bu) \ge \alpha\}.
\]
This condition is equivalent to
\[
N(\bu) - \alpha D(\bu) \ge 0.
\]
Define $H_\alpha(\bu) := N(\bu) - \alpha D(\bu)$. Since $N(\bu)$ is strictly concave and $-\alpha D(\bu)$ is affine (and thus concave), $H_\alpha(\bu)$ is strictly concave. Therefore, the superlevel set $S_\alpha$ is strictly convex. Since all superlevel sets of $R(\bu)$ are strictly convex, $R(\bu)$ is strictly quasiconcave on $\mathcal{U}$.
\end{proof}

\subsection{Proof of Proposition \ref{prop:MNL-complexity}}
\begin{proof}
The proof proceeds in four steps.
\paragraph{(i) Monotonicity and slope bound.}
Recall the value function
\[
\phi(\theta) := \max_{\bp\in\mathcal{P}} G(\bp,\theta).
\]
By the envelope theorem, the value function $\phi$ is differentiable and satisfies
\[
\phi'(\theta)
= -\Bigl(1+\sum_{i=1}^m u_i^\star(\theta)\Bigr),
\qquad
u_i^\star(\theta)=e^{a_i-b_i p_i^\star(\theta)},
\]
where $\bp^\star(\theta) = (p_1^\star(\theta),\ldots,p_m^\star(\theta))$ denotes an optimal solution to the subproblem $\max_{\bp\in\mathcal{P}} G(\bp,\theta)$ for the given value of $\theta$. Since prices satisfy $p_i\le U_i$, we have
$u_i^\star(\theta)\ge e^{a_i-b_i U_i}$ for all $i$. Defining
\[
L_{\mathrm{slope}}
:= 1+\sum_{i=1}^m e^{a_i-b_i U_i},
\]
it follows that
\[
|\phi'(\theta)|\ge L_{\mathrm{slope}}>1
\qquad \forall \theta.
\]
Thus, $\phi(\theta)$ is strictly decreasing with a uniformly bounded slope, implying the existence and uniqueness of the root $\theta^\star$.

\paragraph{(ii) Accuracy of bisection with inexact subproblem solves.}
At each bisection iteration $k$, the algorithm evaluates $\phi(\theta_k)$ at the midpoint $\theta_k$ of the current interval.
Suppose the ECP solver returns an approximate value $\hat{\phi}(\theta_k)$ satisfying
\[
|\hat{\phi}(\theta_k)-\phi(\theta_k)|\le \delta.
\]

If $\mathrm{sign}(\hat{\phi}(\theta_k))=\mathrm{sign}(\phi(\theta_k))$, the bisection update is correct and the true root remains inside the updated interval.
If a sign error occurs, then necessarily $|\phi(\theta_k)|\le \delta$. By the mean value theorem and the slope bound from part (i),
\[
|\theta_k-\theta^\star|
\le \frac{|\phi(\theta_k)|}{|\phi'(\xi)|}
\le \frac{\delta}{L_{\mathrm{slope}}
}
< \delta.
\]
Hence, even in the worst case, the retained point $\theta_k$ is within $\delta$ of $\theta^\star$.

Choosing the bisection termination tolerance and solver accuracy as $\delta=\Theta(\varepsilon)$ guarantees
$|\hat{\theta}-\theta^\star|\le\varepsilon$.

\paragraph{(iii) Complexity of the ECP subproblem.}
Each ECP subproblem contains $O(m)$ exponential cone constraints and $O(m+K)$ variables.
The self-concordant barrier parameter is $\nu=\Theta(m)$.
A primal--dual interior-point method requires
$O\!\left(\sqrt{m}\log\tfrac{1}{\delta}\right)
= O\!\left(\sqrt{m}\log\tfrac{1}{\varepsilon}\right)$
Newton iterations to reach accuracy $\delta$.
Each iteration solves a linear system of dimension $O(m+K)=O(m)$, with cost $O(m^3)$.
Thus, the cost of solving one ECP subproblem is
\[
T_{\mathrm{ECP}}
= O\!\left(m^3\sqrt{m}\log\tfrac{1}{\varepsilon}\right).
\]
\paragraph{(iv) Total timecomplexity.}
The outer bisection loop requires
$O\!\left(\log\tfrac{\theta_{\max}-\theta_{\min}}{\varepsilon}\right)
= O\!\left(\log\tfrac{1}{\varepsilon}\right)$
iterations.
Multiplying by the cost per iteration yields the stated overall time complexity
\[
{O}\!\left(m^{3.5}\log^2\tfrac{1}{\varepsilon}\right).
\]
\end{proof}

\subsection{Proof of Theorem \ref{prop:bnb_complexity}}

\begin{proof}
The total complexity is the product of the number of nodes explored in the B\&B tree and the computational cost of processing each node (solving the ECP relaxation). We analyze each component and the required solver precision $\delta$.

\paragraph{(i) Precision Requirement ($\delta$).}
At each node, we compute a local upper bound ($UB_{node}$) by solving the ECP relaxation numerically with tolerance $\delta$. We also compute a local lower bound ($LB_{node}$) by projecting the solution to feasibility.
The total error at a node is the sum of the structural relaxation error (McCormick gap) and the numerical solver error ($\delta$). To ensure the algorithm terminates when the true gap is within $\epsilon$, we allocate half the error budget to the solver. Thus, we choose the solver tolerance $\delta$ to be of the same order of magnitude as the desired final accuracy $
\epsilon$, i.e.,
    $\delta = \Theta(\epsilon)$
This ensures that numerical noise does not dominate the relaxation gap, preserving the convergence rate of the B\&B scheme.

\paragraph{(ii) Per-Node Complexity (ECP Solver).} At each node, we solve a convex relaxation formulated as an ECP. To determine the computational cost, we analyze the problem structure and the convergence rate of IPM \citep{nesterov1994,andersen2000}.

\textit{Problem Structure:} The relaxation involves $m$ variables for prices (transformed to $u_i$), along with auxiliary variables for the exponential terms. Specifically, the formulation requires:
\begin{itemize}
    \item $m$ exponential cone blocks to model the entropy-like terms $s_i \ge u_i \ln u_i$ (where $(-s_i, u_i, 1) \in \mathcal{K}_{exp}$).
    \item $m$ exponential cone blocks to model the log-constraints $v_i \le \ln u_i$ (where $(v_i, 1, u_i) \in \mathcal{K}_{exp}$).
    \item Linear constraints for price bounds and resource capacities.
\end{itemize}
The total dimension of the problem (variables + constraints) is $n = \Theta(m + K)$.

 The complexity of an IPM depends on the self-concordant barrier parameter $\nu$ of the underlying cone. For a single exponential cone, the barrier parameter is $\nu_{exp} = 3$. Since the problem contains $2m$ such cones, the total barrier parameter is:
\begin{equation*}
    \nu = \sum \nu_{exp} = 3 \times (2m) = \Theta(m).
\end{equation*}

\textit{Iteration Count:} Standard IPM theory dictates that the number of Newton steps required to reduce the duality gap from an initial value to a tolerance $\delta$ is bounded by $O(\sqrt{\nu} \log(1/\delta))$. Substituting our barrier parameter and the chosen tolerance $\delta = \Theta(\epsilon)$:
\begin{equation*}
    N_{iter} = O\left( \sqrt{m} \log\left(\frac{1}{\epsilon}\right) \right)
\end{equation*}

\textit{Arithmetic Cost:} Each Newton step requires forming and solving the KKT system (a linear system of equations) to determine the search direction. For a dense system of size $n$, the computational cost is dominated by matrix factorization (e.g., Cholesky), which requires $O(n^3)$ arithmetic operations.

\textit{Total Node Cost:} Combining the iteration count with the cost per step, the total complexity to process a single node is:
\begin{equation*}
    T_{node} = N_{iter} \times O(n^3) = {O}\left( n^3 \sqrt{m} \log\left(\frac{1}{\epsilon}\right) \right)
\end{equation*}

\paragraph{(iii) Bounding the Node Count in the B\&B tree.}
The core of the complexity analysis relies on bounding the number of spatial subdivisions (nodes) required to reduce the structural relaxation gap below the tolerance $\epsilon$.

\textit{Relaxation Gap Analysis:} Consider a specific customer segment $t \in [T]$. The non-convexity arises from the bilinear term $y_t = \theta_t z_t$. At a given node, let the variable bounds be $[\underline{\theta}_t, \overline{\theta}_t]$ and $[\underline{z}_t, \overline{z}_t]$, with interval widths $\Delta \theta_t = \overline{\theta}_t - \underline{\theta}_t$ and $\Delta z_t = \overline{z}_t - \underline{z}_t$.
The convex relaxation uses the McCormick envelope to approximate the bilinear surface. It is a standard result (McCormick, 1976) that the maximum vertical distance (error) between the true bilinear surface and the convex envelope is bounded by:
\begin{equation*}
    \text{Err}_{Mc} \le \frac{1}{4} \Delta \theta_t \Delta z_t
\end{equation*}
This vertical error in the auxiliary variable $y_t$ propagates to the objective function value. Specifically, the optimality gap contribution from segment $t$ is proportional to this error. To guarantee that the total optimality gap across all $T$ segments is at most $\epsilon/2$ (allocating the other $\epsilon/2$ to solver noise), it is sufficient to enforce that the error for each individual segment is $O(\epsilon/T)$.
Thus, the termination condition for the B\&B procedure on segment $t$ is to reach a node size satisfying:
\begin{equation*}
    \Delta \theta_t \Delta z_t \le C \cdot \frac{\epsilon}{T} = O(\epsilon)
\end{equation*}
where $C$ is a problem-dependent constant related to the bounds of the denominators $z_t$.

\textit{Tree Depth and Leaf Count:} The B\&B algorithm creates children nodes by bisecting the interval of the variable with the largest range. Each bisection reduces the interval width of one variable by half, which reduces the product $\Delta \theta_t \Delta z_t$ by a factor of 2.
Let $A_0$ be the initial area of the domain for segment $t$. The depth of branching $k_t$ required to reduce the area $A_0$ to $O(\epsilon)$ is determined by:
\begin{equation*}
    \frac{A_0}{2^{k_t}} \le O(\epsilon) \implies 2^{k_t} \ge \frac{A_0}{O(\epsilon)} \implies k_t = O\left( \log \frac{1}{\epsilon} \right)
\end{equation*}
However, non-convex optimization generally requires exploring the "surface" of the domain. In the worst case, we cannot prune branches early, and we must explore a full grid covering the domain at this resolution. The number of leaf nodes $N_t$ corresponds to the number of sub-rectangles in this fine grid, which is proportional to $2^{k_t}$:
\begin{equation*}
    N_t \approx 2^{k_t} = O\left( \frac{1}{\epsilon} \right).
\end{equation*}
This confirms that for a single bilinear term, the complexity is polynomial in $1/\epsilon$ (specifically linear).

\textit{Total Nodes:} Since there are $T$ independent bilinear terms (one for each customer segment), the total search space is the Cartesian product of the grids for each segment. The total number of nodes is bounded by the product of the worst-case leaf counts for each segment:
\begin{equation*}
    \mathcal{N}(\epsilon) \approx \prod_{t=1}^T N_t = \prod_{t=1}^T O\left(\frac{1}{\epsilon}\right) = O\left( \left(\frac{1}{\epsilon}\right)^T \right)
\end{equation*}

\paragraph{(iv) Total Complexity.} 
Multiplying the node count by the per-node complexity yields the final bound:
\begin{equation*}
    T_{total} = \mathcal{N}(\epsilon) \times T_{node} = O\left( \left( \frac{1}{\epsilon} \right)^T \right) \cdot {O}\left( n^3 \sqrt{m} \log\left(\frac{1}{\epsilon}\right) \right)
\end{equation*}
When $K=\mathcal{O}(m)$, which is typically the case in practice, the total time complexity can be expressed more compactly as
\[
\mathcal{O}\!\left(\left(\tfrac{1}{\varepsilon}\right)^T m^{3.5}\log\!\left(\tfrac{1}{\varepsilon}\right)\right).
\]
\end{proof}

\section{Detailed B\&B Algorithm}\label{apd:detailed-BnB}

This section provides a detailed description of our B\&B algorithm used to solve the constrained pricing problem under the FMNL model. The algorithm exploits the special structure of the reformulation developed in Section~\ref{sec:pricing-FMNL}, in which all nonconvexity is isolated into a small number of bilinear constraints—one for each customer segment—while all remaining constraints are convex and admit an exponential-cone representation.

\paragraph{Node relaxation and bounding.}
At each node of the B\&B tree, the algorithm solves a convex relaxation obtained by replacing each bilinear equality $y_t = \theta_t z_t$ with its McCormick envelope over the current bounds of $(\theta_t, z_t)$. Together with the exponential-cone constraints that encode the transformed pricing model, this yields a convex ECP. Solving this relaxation provides a valid upper bound $U_N$ on the optimal objective value attainable within the corresponding subregion of the search space. If the relaxation is infeasible or if $U_N$ does not exceed the best known feasible objective value $R^*$, the node is pruned immediately.

\paragraph{Incumbent update and feasibility check.}
Whenever a node relaxation is solved (by MOSEK), the algorithm extracts a candidate price vector by mapping the relaxed solution back to prices via $\hat{p}_i = -\ln(\hat{u}_i)/b_i$ for all $i\in [m]$. If this candidate improves upon the current incumbent, the global lower bound $R^*$ and the incumbent solution $\bp^*$ are updated. To assess whether the relaxation is sufficiently tight, the algorithm computes the maximum violation of the bilinear constraints,
\[
\delta_{\max} = \max_t \left| \hat{y}_t - \hat{\theta}_t \hat{z}_t \right|.
\]
If $\delta_{\max} \le \varepsilon$, the current solution is guaranteed to be globally $\varepsilon$-optimal, and the algorithm terminates.

\paragraph{Branching strategy.}
If the bilinear constraints are not yet satisfied to the desired tolerance, the algorithm branches on the most violated segment. Specifically, it identifies the index $t^*$ attaining $\delta_{\max}$ and selects either $\theta_{t^*}$ or $z_{t^*}$ (whichever has the larger current interval) for branching. The chosen variable is bisected, generating two child nodes with tightened bounds. This targeted branching strategy focuses computational effort on the dominant source of nonconvexity and avoids unnecessary subdivision along dimensions that are already well resolved.

\paragraph{Node selection and implementation details.}
Nodes are explored using a BFS strategy, prioritizing nodes with larger upper bounds. To improve efficiency, the algorithm employs a persistent solver architecture: instead of rebuilding the optimization model at each node, only variable bounds are updated before re-solving the ECP. This significantly reduces overhead and allows warm-starting of the IP solver. Algorithm~\ref{alg:bb_compact} summarizes the full procedure.

\paragraph{Execution flow.}
Figure~\ref{fig:bb-flow-queue} illustrates the interaction between the B\&B tree and the priority queue. The left panel depicts the branching structure and pruning decisions, while the right panel shows the evolution of the node queue under BFS. The figure highlights how strong relaxations and early incumbent updates can lead to aggressive pruning, allowing the algorithm to terminate after exploring only a small fraction of the full search tree.

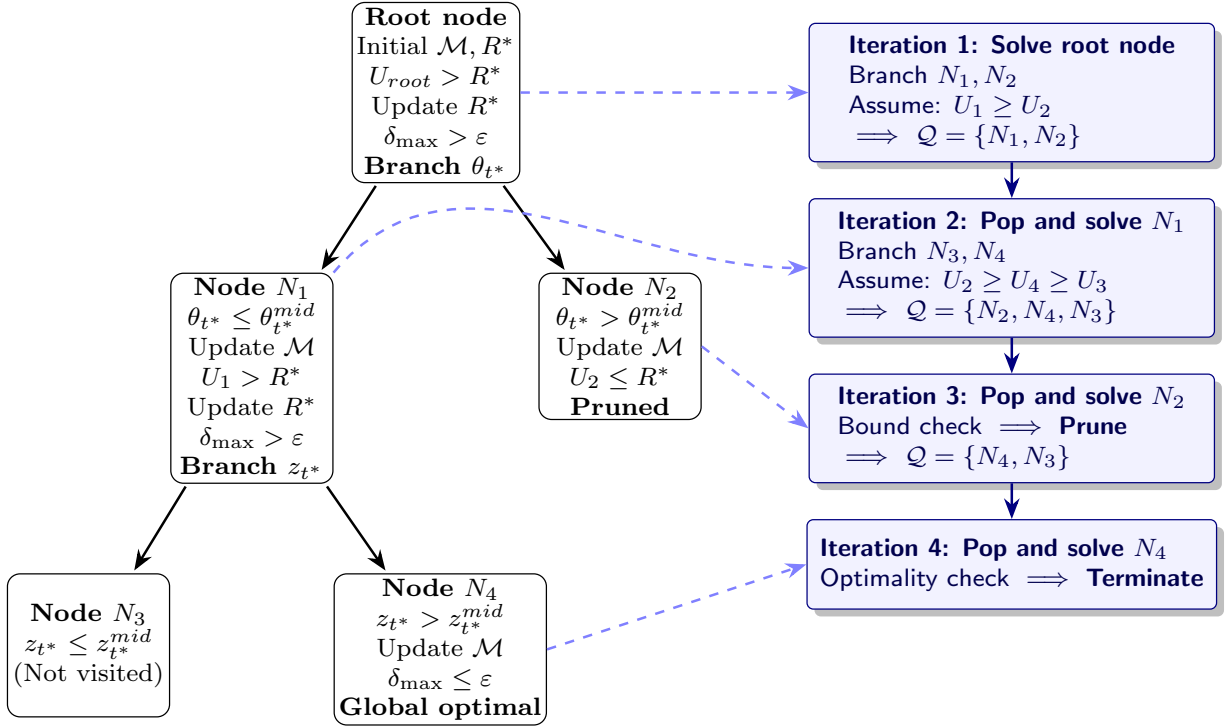
\begin{figure}[htb]
\centering
% Resizebox scales the figure to column width. 
% By reducing the internal height, the aspect ratio changes 
% so it takes up less vertical space on the page.
\resizebox{\columnwidth}{!}{%
\begin{tikzpicture}[
  >=Stealth,
  node distance=8mm,
  % --- UNIFIED FONT STYLE ---
  every node/.style={font=\scriptsize},
  % --- LEFT STYLE (Compact Tree) ---
  tree_node/.style={
    draw, rectangle, rounded corners,
    align=center, 
    inner sep=1.5pt,       % Reduced padding
    minimum width=1.8cm,
    minimum height=1.6cm   % Reduced minimum height (was 2.2cm)
  },
  % --- RIGHT STYLE (Wide Queue) ---
  queue_node/.style={
    draw=blue!50!black, rectangle, rounded corners=2pt,
    fill=blue!5,
    align=left, 
    font=\scriptsize\sffamily, 
    inner sep=4pt,         % Reduced padding
    minimum width=4.5cm,
    minimum height=1.0cm,  % Reduced minimum height
    text=blue!30!black,
    drop shadow
  },
  edge/.style={->, thick, shorten >=1pt, shorten <=1pt},
  logic_edge/.style={->, dashed, color=blue!50, thick}
]

% ==========================================
% COLUMN 1: THE B&B TREE
% ==========================================

% --- Root Node (N0) ---
\node[tree_node] (root) {
    \textbf{Root node}\\
    Initial $\mathcal{M}, R^*$\\
    $U_{root} > R^*$\\
    Update $R^*$\\
    $\delta_{\max} > \varepsilon$\\
    \textbf{Branch $\theta_{t^*}$}
};

% --- Level 1 ---
% REDUCED VERTICAL DISTANCE: 20mm -> 10mm
\node[tree_node] (n1) [below left=10mm and 2mm of root] {
    \textbf{Node $N_1$}\\
    $\theta_{t^*} \le \theta_{t^*}^{mid}$\\
    Update $\mathcal{M}$\\
    $U_1 > R^*$\\
    Update $R^*$\\
    $\delta_{\max} > \varepsilon$\\
    \textbf{Branch $z_{t^*}$}
};

\node[tree_node] (n2) [below right=10mm and 2mm of root] {
    \textbf{Node $N_2$}\\
    $\theta_{t^*} > \theta_{t^*}^{mid}$\\
    Update $\mathcal{M}$\\
    $U_2 \le R^*$\\
    \textbf{Pruned}
};

% --- Level 2 ---
% REDUCED VERTICAL DISTANCE: 20mm -> 10mm
\node[tree_node] (n4) [below left=10mm and 0mm of n1] {
    \textbf{Node $N_3$}\\
    $z_{t^*} \le z_{t^*}^{mid}$\\
    (Not visited)
};

\node[tree_node] (n3) [below right=10mm and 0mm of n1] {
    \textbf{Node $N_4$}\\
    $z_{t^*} > z_{t^*}^{mid}$\\
    Update $\mathcal{M}$\\
    $\delta_{\max} \le \varepsilon$\\
    \textbf{Global optimal}
};

% --- Tree Edges ---
\draw[edge] (root) -- (n1);
\draw[edge] (root) -- (n2);
\draw[edge] (n1) -- (n4);
\draw[edge] (n1) -- (n3);

% ==========================================
% COLUMN 2: QUEUE LOGIC FLOWCHART
% ==========================================

% Position q1 relative to root
\node[queue_node] (q1) [right=3.2cm of root] {
    \textbf{Iteration 1: Solve root node}\\
    Branch $N_1, N_2$\\
    Assume: $U_{1} \ge U_{2}$\\
    $\implies$ $\mathcal{Q}=\{N_1,N_2\}$
};

% REDUCED VERTICAL GAP: 0.8cm -> 0.4cm
\node[queue_node] (q2) [below=0.4cm of q1] {
    \textbf{Iteration 2: Pop and solve $N_1$}\\
    Branch $N_3, N_4$\\
    Assume: $U_{2} \ge U_{4} \ge U_3$\\
    $\implies$ $\mathcal{Q}=\{N_2, N_4, N_3\}$
};

\node[queue_node] (q3) [below=0.4cm of q2] {
    \textbf{Iteration 3: Pop and solve $N_2$}\\
    Bound check $\implies$ \textbf{Prune}\\
    $\implies$ $\mathcal{Q}=\{N_4, N_3\}$
};

\node[queue_node] (q4) [below=0.4cm of q3] {
    \textbf{Iteration 4: Pop and solve $N_4$}\\
    Optimality check $\implies$ \textbf{Terminate}
};

% --- Logic Flow Arrows (Vertical) ---
\draw[->, thick, blue!50!black] (q1) -- (q2);
\draw[->, thick, blue!50!black] (q2) -- (q3);
\draw[->, thick, blue!50!black] (q3) -- (q4);

% --- Connection Arrows (Horizontal/Curved) ---

% 1. Root -> Q1
\draw[logic_edge] (root.east) -- (q1.west);

% 2. N1 -> Q2
% Adjusted curve angle (out=-50) to dive steeper under N2 due to tight spacing
\draw[logic_edge] (n1.north east) to[out=50, in=180] (q2.west);

% 3. N2 -> Q3
\draw[logic_edge] (n2.east) -- (q3.west);

% 4. N3 -> Q4
\draw[logic_edge] (n3.east) -- (q4.west);

\end{tikzpicture}%
}
\caption{B\&B execution flow. \textbf{Left:} The tree structure. \textbf{Right:} Queue operations.}
\label{fig:bb-flow-queue}
\end{figure}

Overall, the proposed B\&B algorithm systematically refines the feasible region only along the bilinear dimensions induced by customer heterogeneity, while leveraging convex optimization to efficiently bound and prune the search space. This structure is the key to achieving global $\varepsilon$-optimality with computational complexity that is exponential only in the number of customer segments and polynomial in all other problem dimensions.

\section{Additional Experiments}\label{apd:additional-experiments}
We report additional numerical experiments that further evaluate the proposed methods. 
\subsection{Constrained Pricing under MNL: B\&B, Bisection--ECP, and Gradient-Based Methods}
\label{apd:experiment-MNL}
\begin{figure}[htb]
    \centering
    \includegraphics[width=\linewidth]{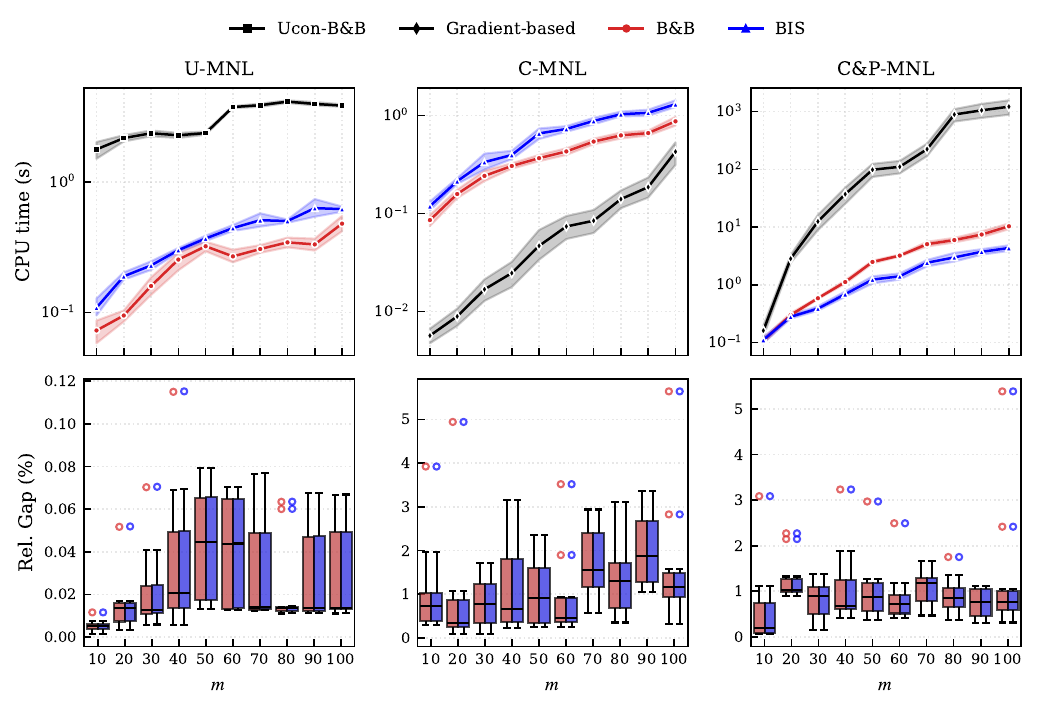}
    \caption{Our proposed method (B\&B and BIS) vs. the Ucon-B\&B method in \cite{ruben2022mnsc} and the gradient-based approach on different settings of the pricing problem under MNL.}
    \label{fig:MNL}
\end{figure}

Figure~\ref{fig:MNL} describes the performance results of all methods on the price optimization problem under MNL model, with line plots representing a 95\% confidence interval. The first row illustrates the runtime across all methods, while the second row displays the relative gap of our proposed methods (B\&B and BIS) against the baselines (Ucon-B\&B and gradient-based).

In the first column, which covers the unconstrained MNL setting, denoted by ``U-MNL'', the gradient-based method is excluded as it serves as a warm start for the Ucon-B\&B. The results show that the Ucon-B\&B is notably slower than our exact methods, with our approaches providing slightly better optimal solutions, up to 0.1\% gap. The second column evaluates the problem with capacity constraints, denoted by ``C-MNL''. While the gradient-based method’s runtime is slightly lower, it suffers from a significant optimality gap. In contrast, B\&B and BIS consistently find optimal solutions of all C-MNL instances, improving solution quality by up to 5\% within a few seconds—typically staying well below 1 second.

The final column explores the most complex scenario involving both pairwise-price and capacity constraints, denoted by ``C\&P-MNL''. Here, the addition of pairwise price constraints significantly hinders the runtime of the gradient-based method, whereas our methods remain robust and less affected. As shown in the final gap figure, the exact solutions provided by B\&B and BIS outperform the gradient-based near-optimal solutions by a margin of up to 5\%, verifying their superior precision and efficiency in constrained environments.

\subsection{Unconstrained Pricing under FMNL}\label{apd:ucon-FMNL}
\begin{figure}[htb]
    \centering
    \includegraphics[width=\linewidth]{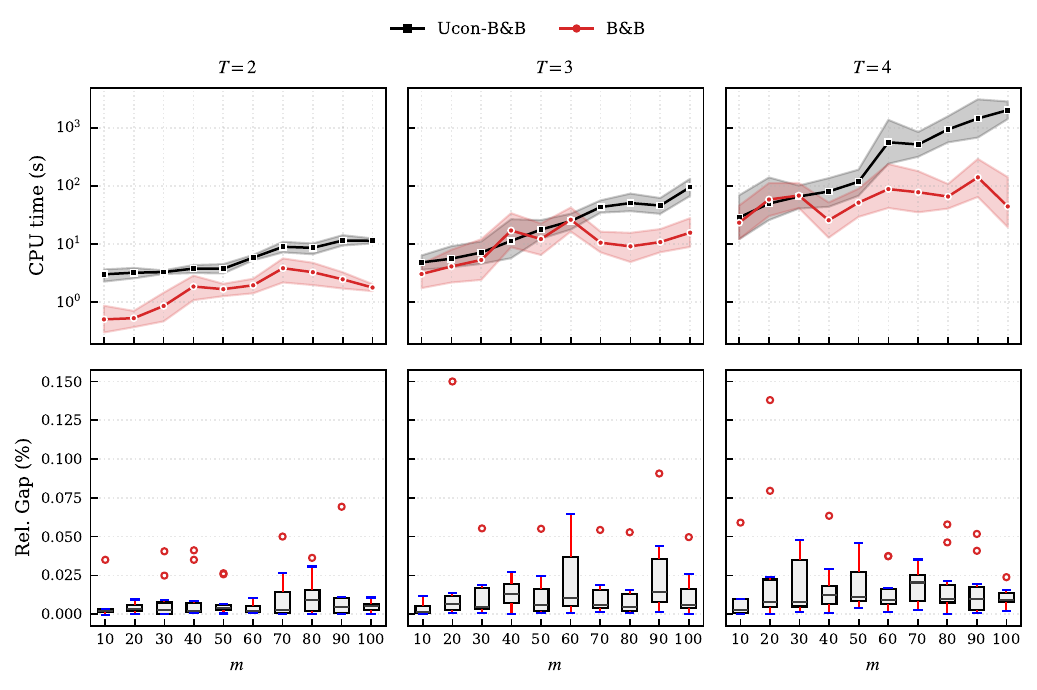}
    \caption{Our B\&B method vs. Ucon-B\&B algorithm on the U-FMNL dataset.}
    \label{fig:obj_time_FMNL}
\end{figure}
We next compare our proposed B\&B algorithm with the Ucon-B\&B method of \citep{ruben2022mnsc} for unconstrained pricing under FMNL demand. Figure~\ref{fig:obj_time_FMNL} summarizes the results in terms of both computational efficiency and solution quality. Since the benchmark B\&B method does not impose a runtime limit, we likewise allow our B\&B algorithm to run without a time cap in order to ensure a fair comparison.

As shown in the first row of Figure~\ref{fig:obj_time_FMNL}, the proposed B\&B method consistently requires less CPU time than the Ucon-B\&B across almost all tested numbers of products $m$ and customer segments $T\in\{2,3,4\}$. The performance gap becomes particularly pronounced as problem complexity increases. For example, when $T=4$, the benchmark B\&B frequently exceeds one hour of runtime for instances with $m\in\{90,100\}$ products. In contrast, our method reliably computes globally optimal solutions in under 1{,}000 seconds for the same instances, demonstrating substantially better scalability.

The second row of Figure~\ref{fig:obj_time_FMNL} reports the optimality gaps. Across all instances, our B\&B method achieves gaps that are comparable to or slightly smaller than those obtained by the benchmark method, with improvements of up to $0.15\%$. Even for the largest instances considered, our algorithm maintains tight and stable optimality gaps, indicating robust numerical performance.

Overall, these results confirm that our B\&B method is not only significantly faster but also at least as accurate—and in some cases more accurate—than the current state-of-the-art. This empirical evidence aligns with the theoretical complexity analysis in Theorem~\ref{prop:bnb_complexity}, which predicts a lower polynomial dependence on the problem dimension for our formulation.

\subsection{Pricing under FMNL with Capacity Constraints}
\label{apd:FMNL-C-constraints}
\begin{figure}[!h]
    \centering
    \includegraphics[width=\linewidth]{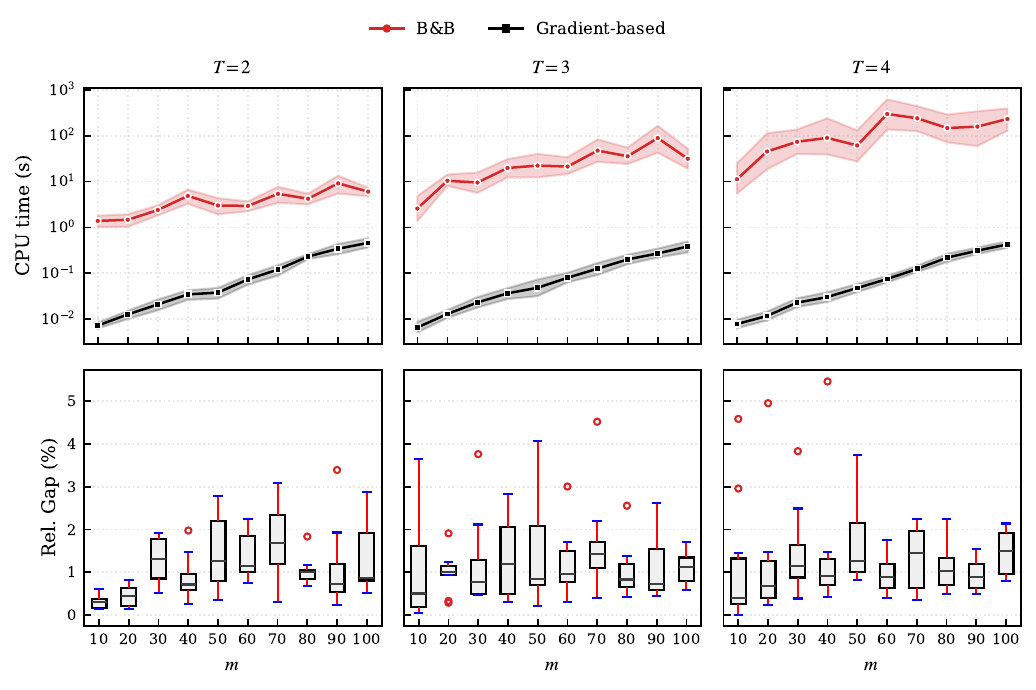}
    \caption{Our B\&B method vs. the gradient-based algorithm on the C-FMNL dataset.}
    \label{fig:obj_time_FMNL_Price}
\end{figure}
We next report numerical results for instances with only capacity and bound constraints, referred to as the C-FMNL setting. Figure~\ref{fig:obj_time_FMNL_Price} summarizes the performance of the proposed B\&B method under C-FMNL across different numbers of customer segments ($T=2,3,4$). In this setting, the gradient-based local search method achieves very low runtimes—often under one second—reflecting its lightweight computational nature. However, this speed comes at the expense of solution quality, as the gradient-based approach does not provide global optimality guarantees. As shown by the relative optimality gap plots, the gradient-based method can be up to 5\% worse than the true optimum identified by our B\&B algorithm.

In contrast, although the B\&B method requires higher computational effort, it consistently identifies globally optimal solutions, with CPU times remaining below 1{,}000 seconds even in the most challenging cases. Importantly, incorporating capacity constraints does not significantly degrade the performance of the proposed B\&B approach. The observed runtimes for the C-FMNL and U-FMNL settings are nearly identical: a few seconds for $T=2$, tens of seconds for $T=3$, and a few hundred seconds for $T=4$.

From a practical standpoint, these results suggest that capacity constraints—while essential for realistic pricing decisions—do not materially increase the computational burden of global optimization when handled within the proposed framework. This robustness makes the B\&B method well suited for real-world pricing applications where capacity limits are ubiquitous and solution quality is critical.

\subsection{Constrained Pricing under FMNL,  Capacity and Pairwise Constraints}\label{appd:HC_FMNL}
In this section, we points out a saturation in the trade-off between the accuracy of the gradient-based algorithm and its short runtime. While the tolerance $\epsilon$ of the gradient-based method can be reduced to obtain results closer to the optimal solution in C-FMNL cases because its CPU time is lower than 1 second, doing that in the C\&P-FMNL setting may cause the gradient-based algorithm to reach the time limit.

\begin{figure}[htb]
    \centering
    \includegraphics[width=\linewidth]{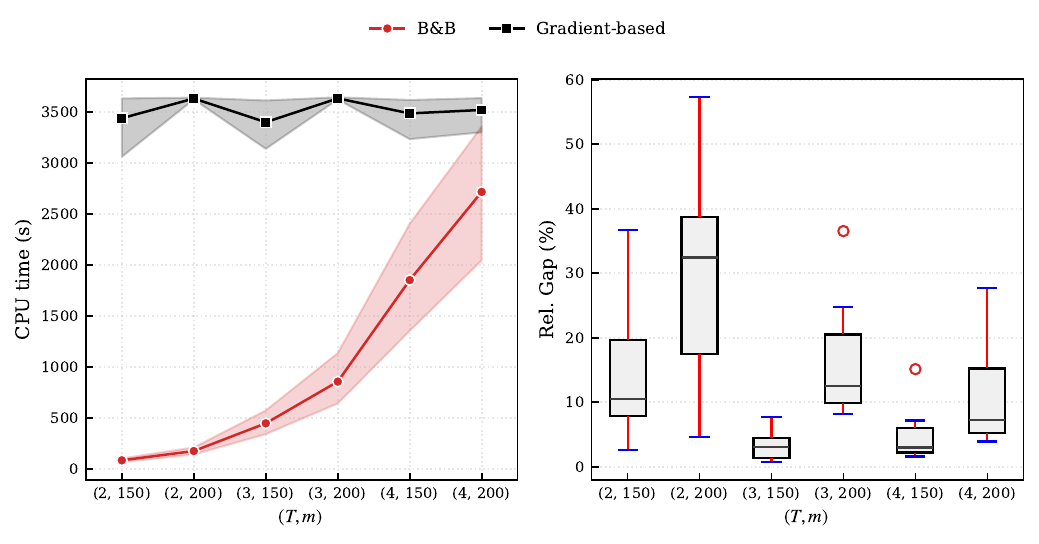}
    \caption{Our B\&B method vs. the gradient-based algorithm on the $m$-large C\&P-FMNL instances.}
    \label{fig:large_m}
\end{figure}

In the C\&P-FMNL setting with a large number of customers, as shown in Figure~\ref{fig:large_m}, the performance comparison between the B\&B method and the gradient-based method reveals a stark contrast in both computational efficiency and solution quality. As illustrated in the CPU time chart, even with an optimality tolerance of $\epsilon = 0.01$, the gradient-based method consistently reaches the predefined time limit of 3,600 seconds for nearly every instance. In contrast, the B\&B method demonstrates superior scalability, successfully solving all instances to optimality for complexity levels $T \leq 3$ and completing 14 out of 20 instances even at the highest complexity of $T=4$. The gap figure highlights a massive disparity in solution quality improvement. The proposed B\&B method achieves a significantly higher objective value, demonstrating a gap of over 50\% in certain instances compared to the gradient-based method. For example, at $(T, m) = (2, 200)$, the B\&B method shows a median improvement of approximately 30\%, with upper bounds reaching nearly 60\%. This substantial gap underscores that the gradient-based method not only struggles with computational limits in complex, large-scale scenarios but also fails to provide high-quality near-optimal solutions, whereas the B\&B method remains both reliable and highly precise.

\subsection{Node Exploration in B\&B Algorithm}
\begin{figure}[!h]
    \centering
    \includegraphics[width=\linewidth]{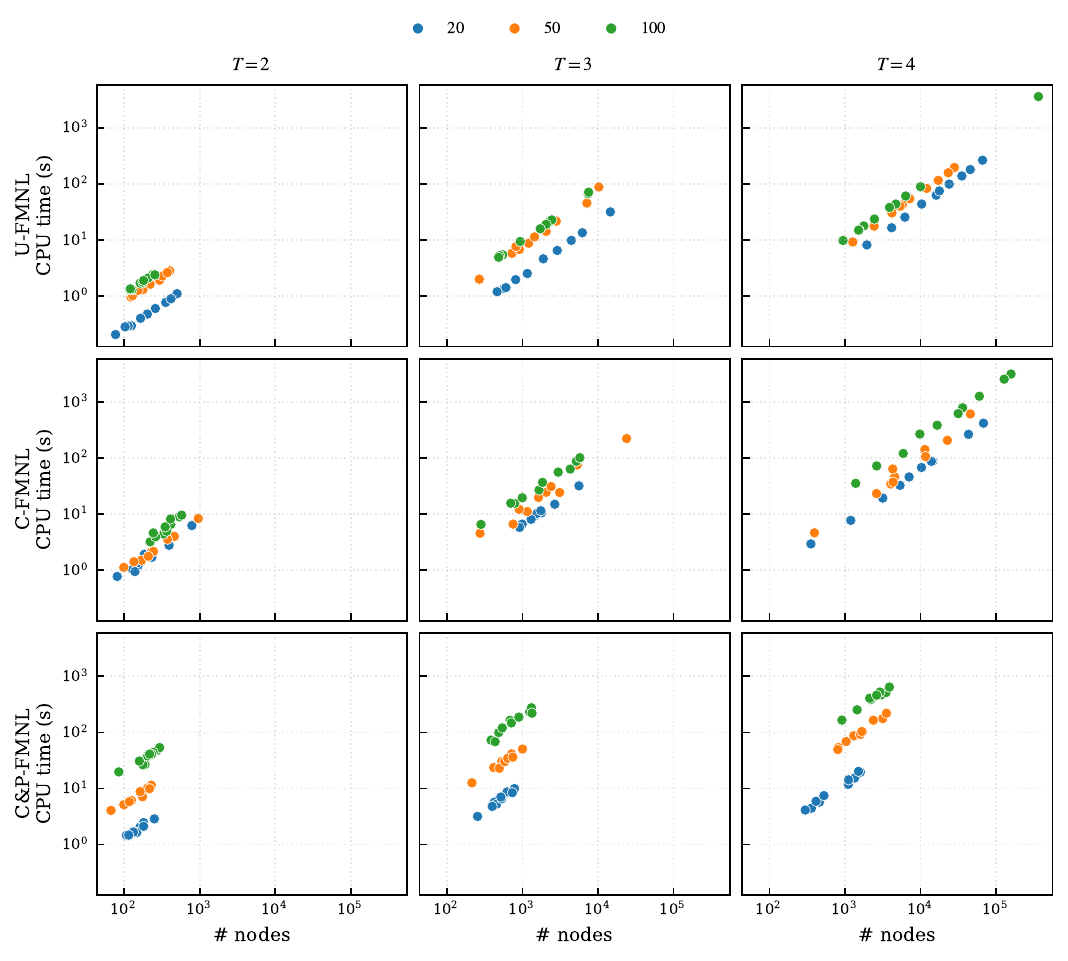}
    \caption{Node-CPU time relationship in the proposed B\&B method on different settings of the pricing problem under FMNL.}
    \label{fig:time_node}
\end{figure}
We report the number of nodes explored by the proposed B\&B algorithm under different experimental settings. Figure~\ref{fig:time_node} reveals a strong log--log correlation between CPU time and the number of explored nodes across all settings, confirming that the size of the B\&B search tree is the primary driver of computational effort. Within each setting, the dispersion of points reflects the effect of problem scale: instances with larger product sets ($m$) and more customer segments ($T$) tend to cluster in the upper-right region, indicating both higher node counts and longer runtimes.

A comparison across settings highlights distinct computational behaviors. The U-FMNL and C-FMNL settings exhibit similar patterns, with most instances for $T\in\{3,4\}$ requiring between $10^3$ and $10^5$ nodes to converge. The C-FMNL setting incurs slightly higher CPU time per node due to the additional overhead associated with enforcing capacity constraints. In contrast, the C\&P-FMNL setting benefits substantially from the inclusion of pairwise pricing constraints, which significantly tighten the feasible region and enable more effective pruning of the search tree. As a result, these instances typically converge after exploring fewer than $10^4$ nodes—often an order of magnitude fewer than in the other settings.

Although solving each node in the C\&P-FMNL setting is computationally more expensive, the pronounced reduction in the number of explored nodes more than offsets this cost. Consequently, problems in this setting are solved efficiently within comparable time limits to the U-FMNL and C-FMNL cases. These results illustrate an important insight: incorporating realistic pricing constraints not only improves solution relevance but can also enhance computational efficiency by strengthening relaxations and accelerating convergence.

\end{document}

\end{document}